\documentclass{conm-p-l}
\usepackage{amsmath,amssymb,amscd,amsxtra}
\usepackage{calrsfs}
\usepackage{bm}

\theoremstyle{plain}
\newtheorem{theo}{Theorem}[section]

\newtheorem{zerotheo}{Theorem}

\newtheorem{prop}[theo]{Proposition}
\newtheorem{lemm}[theo]{Lemma}
\newtheorem{coro}[theo]{Corollary}

\newtheorem{sub}[theo]{Sublemma}

\newtheorem{zerocoro}{Corollary}

\theoremstyle{definition}
\newtheorem*{defi}{Definition}

\theoremstyle{remark}
\newtheorem{rema}[theo]{Remark}
\newtheorem*{note}{Note}

\renewcommand{\labelenumi}{(\roman{enumi})}

\renewcommand{\u}{u_i^I}

\newcommand{\V}{\mathcal V}

\newcommand{\hatEst}{\hat{\mathcal{E}}_{st}}

\newcommand{\field}[1]{\mathbb{#1}}
\newcommand{\C}{\field{C}}

\newcommand{\Q}{\field{Q}}
\newcommand{\R}{\field{R}}
\newcommand{\Z}{\field{Z}}

\newcommand{\phist}{\phi_{st}}
\newcommand{\hatvar}{\hat{\varphi}}
\newcommand{\hatvarst}{\hat{\varphi}_{st}}
\newcommand{\hatvarepst}{\hat{\varepsilon}_{st}}

\newcommand{\hatvarvst}{\hat{\varphi}_{st}^v}

\newcommand{\hatH}{\hat{H}}

\newcommand{\sigmn}{\Sigma^{(n)}}
\newcommand{\sigmk}{\Sigma^{(k)}}
\newcommand{\sigmone}{\Sigma^{(1)}}
\newcommand{\sigmpr}{\Sigma^\prime}
\newcommand{\sigmprone}{\Sigma^{\prime(1)}}
\newcommand{\sigmprn}{\Sigma^{\prime(n)}}
\newcommand{\delpr}{\Delta^\prime}
\newcommand{\Cpr}{C^\prime}
\newcommand{\Ipr}{I^\prime}
\newcommand{\Jpr}{J^\prime}
\newcommand{\ipr}{i^\prime}
\newcommand{\jpr}{j^\prime}
\newcommand{\Vpr}{\mathcal{V}^\prime}
\newcommand{\xipr}{\xi^\prime}

\newcommand{\sigmprpr}{\Sigma^{\prime\prime}}
\newcommand{\delprpr}{\Delta^{\prime\prime}}

\newcommand{\Jprpr}{J^{\prime\prime}}

\newcommand{\xiprpr}{\xi^{\prime\prime}}

\newcommand{\delstar}{\Delta_*}
\newcommand{\Vstar}{\mathcal{V}_*}
\newcommand{\xistar}{\xi_*}
\newcommand{\delstartil}{\tilde{\Delta}_*}
\newcommand{\Vstartil}{\tilde{\mathcal{V}}_*}
\newcommand{\xistartil}{\tilde{\xi}_*}
\newcommand{\sigmstar}{\Sigma_*}
\newcommand{\sigmstartil}{\tilde{\Sigma}_*}

\newcommand{\Cstartil}{\tilde{C}_*}

\newcommand{\delprstar}{\Delta^\prime_*}
\newcommand{\Vprstar}{\mathcal{V}^\prime_*}
\newcommand{\xiprstar}{\xi^\prime_*}
\newcommand{\sigmprstar}{\Sigma^\prime_*}
\newcommand{\Cstar}{C_*}
\newcommand{\Cprstar}{C^\prime_*}

\newcommand{\delprprstar}{\Delta^{\prime\prime}_*}

\newcommand{\xiprprstar}{\xi^{\prime\prime}_*}
\newcommand{\sigmprprstar}{\Sigma^{\prime\prime}_*}
\newcommand{\Cprprstar}{C^{\prime\prime}_*}

\newcommand{\tilG}{\tilde{G}}
\newcommand{\tilT}{\tilde{T}}

\newcommand{\OmeI}{\Omega^I}
\newcommand{\OmeIpr}{\Omega^{\Ipr}}
\newcommand{\omeI}{\omega^I}
\newcommand{\omeIpr}{\omega^{\Ipr}}

\newcommand{\barh}{\bar{h}}

\newcommand{\uv}{\langle u_i^I,v\rangle}

\def\l{\langle}
\def\r{\rangle}

\newcommand{\Proj}{\field{P}}

\newcommand{\img}{\sqrt{-1}}

\title{Invariance Property of Orbifold Elliptic Genus for Multi-Fans}
\author{Akio Hattori}
\address{Graduate School of Mathematical Science, University of Tokyo,
Tokyo, Japan}
\email{hattori@ms.u-tokyo.ac.jp}

\begin{document}
\subjclass[2000]{Primary: 58J26, 57R20, 14M25.}

\maketitle



\section{Introduction}
\label{sec:1}
The complex elliptic genus was first introduced by Witten \cite{Wit}
and then studied by several authors such as Hirzebruch \cite{Hir},
Bott and Taubes \cite{BT} mainly in connection with its rigidity
property. It was further generalized in two ways; one way to complex
orbifolds and the other to singular projective varieties.

Generalization to singular varieties was given by Borisov-Libgober
\cite{BL2}, \cite{BL3}. They called it singular elliptic genus;
it is defined for Kawamata-log-terminal pairs $(X,D)$ of a variety $X$ and
a $\Q$-divisor $D$.
We shall denote it by $Ell_{sing}(X,D)$. It has an invariant
property with respect to blow-ups. Namely, if $f:\tilde{X}\to X$ is
a blow-up along a non-singular locus in $X$ which is normal crossing to
$Supp(D)$ and $\tilde{D}$ is a divisor on $\tilde{X}$ such that
\begin{equation*}\label{eq:Klift}
 K_{\tilde{X}}+\tilde{D}=f^*(K_X+D),
\end{equation*}
then
\begin{equation}\label{eq:pushEll}
Ell_{sing}(\tilde{X}, \tilde{D})=Ell_{sing}(X,D).
\end{equation}

The formula \eqref{eq:pushEll} is related to the work of Totaro \cite{T}.
He showed that the Chern numbers that can be extended to singular varieties,
compatibly with $IH$-small resolutions, are at most linear combinations of
the coefficients of the elliptic genus. The formula \eqref{eq:pushEll}
implies that the elliptic genus can in fact be defined for projective
varieties with Kawamata-log-terminal singularities, complementing the
result of Totaro. In particular, if $X$ is such a variety and
$f:\tilde{X}\to X$ is a crepant resolution, then
\begin{equation}\label{eq:Ellcrepant}
 Ell(\tilde{X})=Ell_{sing}(X),
\end{equation}
where $Ell_{sing}(X)=Ell_{sing}(X,0)$ and $Ell(\tilde{X})$ denotes the
ordinary elliptic genus of $\tilde{X}$.

Orbifold elliptic genus was introduced by Borizov-Libgober \cite{BL1}
for global quotient complex orbifolds, and then was generalized to general
complex (or more generally to stably complex) orbifolds by Don-Liu-Ma
\cite{DLM}. Borizov and Libgober in \cite{BL3} also defined orbifold
elliptic genus $Ell_{orb}(X,D,G)$ where $G$ is a finite group and $(X,D)$
is a Kawamata-log-terminal $G$-normal pair with smooth $X$ and showed
a similar formula to \eqref{eq:pushEll}. They also proved a formula
\begin{equation}\label{eq:Ellquot}
 Ell_{orb}(X,D,G)=Ell_{sing}(X/G,D_{X/G})
\end{equation}
for a suitably defined divisor $D_{X/G}$. For example, when $D=0$,
$D_{X/G}$ is
given by the following formula. Let $\pi:X\to X/G$ be the quotient map and
$E=\sum(a_i-1)E_i$ the ramification divisor of $\pi$ where the sum runs
over prime divisors $E_i$. Then $D_{X/G}$ is given by
\begin{equation}\label{eq:ramif}
 D_{X/G}=\sum\frac{a_i-1}{a_i}\pi(E_i).
\end{equation}

A similar formula was already given by Batyrev \cite{Bat} for
$E$-function. The $E$-function is a generalization of Hirzebruch's
$\chi_y$-genus to singular varieties. The elliptic genus is also
a generalization of $\chi_y$-genus. Suppose that the fixed point set
$X^g$ of the action of each element $g\in G$ has codimension at least
two. Then the ramification divisor is trivial, and the formula
\eqref{eq:Ellquot} for $D=0$ reduces to
\begin{equation}\label{eq:Ellquot2}
 Ell_{orb}(X,G)=Ell_{sing}(X/G).
\end{equation}
If moreover $X/G$ has a crepant resolution $\tilde{X}\to X/G$, then we get
\[ Ell_{orb}(X,G)=Ell(\tilde{X}), \]
by \eqref{eq:Ellcrepant}. This sort of results goes back to \cite{DHVW}
where stringy Euler number is considered instead of singular elliptic
genus, and is related to an
observation of Mckay concerning the relation between minimal resolutions
of quotient singularities $\C^2/G$ and the representations of $G$.

Borizov and Libgober define in \cite{BL3} not only the genus but a class
$\mathcal{E}\ell\ell_{orb}(X,D,G)$ for $G$ normal pair $(X,D)$ in the Chow
ring $A_*(X)$ in such a way that the elliptic genus $Ell_{orb}(X,D,G)$
becomes the degree of the top component of $\mathcal{E}\ell\ell_{orb}(X,D,G)$.
They then prove the functorial property
\begin{equation}\label{eq:pushEll2}
f_*\mathcal{E}\ell\ell_{orb}(\tilde{X},\tilde{D},G)=
\mathcal{E}\ell\ell_{orb}(X,D,G),
\end{equation}
where $(\tilde{X},\tilde{D})$ and $(X,D)$ are $G$-normal pairs related
together as in \eqref{eq:pushEll}. The formula \eqref{eq:pushEll2} is
sometimes called change of variables formula. The main result of \cite{BL3}
is the following formula
\begin{equation}\label{eq:Ellquot3}
 \pi_*\mathcal{E}\ell\ell_{orb}(X,D,G)=\mathcal{E}\ell\ell_{sing}(X/G,D_{X/G}),
\end{equation}
where $\mathcal{E}\ell\ell_{sing}(X/G,D_{X/G})$ is a class in $A_*(X/G)$
defined in a similar way such that the degree of its top component
coincides with $Ell_{sing}(X/G,D_{X/G})$. \eqref{eq:Ellquot} immediately
follows from \eqref{eq:Ellquot3}. There is an equivariant
version due to Waelder \cite{Wae1} which is a good reference for this
subject. See also \cite{Wae2}.

The (complex) orbifold elliptic genus is defined for compact, stably almost
complex orbifolds in general. We shall write it $\hatvar(X)$.
It essentially depends on orbifold
structures. There are many examples of orbifolds with the same underlying
space but with different orbifold structures and different orbifold elliptic
genus. This phenomenon is related to the above formulae \eqref{eq:Ellquot},
\eqref{eq:Ellquot2} and \eqref{eq:Ellquot3}.
The singular elliptic genus is defined by using resolution of singularities.
One might hope to get a direct, topological definition of the singular
elliptic genus which can be extended to a larger class of singular spaces.
In the case of orbifolds one already has orbifold elliptic genus. One
wants to get a suitable notion of $\Q$-divisors and Chow ring which can
be applied to formulate change of variables formula.

There is a class of orbifolds, called torus orbifolds, in which
one can build a satisfactory theory. A torus orbifold $X$ of
dimension $2n$ is, roughly speaking, a $2n$-dimensional compact
stably almost complex orbifold with an action of an
$n$-dimensional torus $T$. Torus orbifolds can be considered as
topological counterparts of $\Q$-factorial toric varieties. For a
$\Q$-factorial toric variety $X$ of dimension $n$ there is
associated a simplicial fan $\Delta$ in an $n$-dimensional lattice
$L$. Algebro-geometric properties of a toric variety are
translated to those of the fan associated to the variety. To each
edge (one dimensional cone) of $\Delta$ there corresponds an
irreducible $T$-divisor $D_i$ and there is an exact sequence
\begin{equation}\label{eq:Chow}
 0\longrightarrow L^* \longrightarrow \bigoplus_i \Z\cdot D_i
 \longrightarrow A_{n-1}(X)\longrightarrow 0,
\end{equation}
where $A_{n-1}(X)$ is the $(n-1)$-th Chow group of $X$, see e.g.
\cite{Ful}. We note here that the dual lattice $L^*$ can be identified
with the second cohomology $H^2(BT)$ of a classifying space of $T$,
and the middle term of \eqref{eq:Chow} is identified with
the second equivariant cohomology $H_T^2(X;\Q)$ after tensored by $\Q$.
It is also identified with the degree two part of the
Stanley--Reisner ring of $\Delta$ considered as a simplicial set.

To a torus orbifold there is associated a simplicial multi-fan, an
analogue of fan, and an integral edge vector is assigned to each
$1$-dimensional cone of the multi-fan, see \cite{M}, \cite{HM1}.
These vectors are not primitive in general and they reflect the orbifold
structure of the torus orbifold whereas one always takes primitive vectors
when dealing with toric varieties.
Moreover divisors over a torus orbifold can be defined as homogeneous
elements of degree two in the Stanley-Reisner ring of the simplicial
set associated to the multi-fan.

The (equivariant, stabilized) orbifold elliptic genus
$\hatvarst(\Delta,\V,\xi)$
is defined for triples of a simplicial multi-fan $\Delta$,
a set of edge vectors $\V$ and a $\Q$-divisor $\xi$. We can
go further to define orbifold elliptic class $\hatEst(\Delta,\V,\xi)$ of
such triples in the Stanley-Reisner ring with $\Q$-coefficients.
When $\Delta$ is the multi-fan associated to a torus orbifold $X$ the
orbifold elliptic genus and orbifold elliptic class are the invariants
of $X$ and the divisor $\xi$. The push-forward from the
Stanley-Reisner ring to the complex numbers $\C$ sends the orbifold
elliptic class to the orbifold elliptic genus.

Birational morphisms between multi-fans can be defined in such a
way that they correspond to geometric birational morphisms between
toric varieties. Moreover if $f:\delpr\to \Delta$ is a birational
morphism, and $\V$ and $\Vpr$ are sets of generating edge vectors
for $\Delta$ and $\delpr$ respectively, then the pull-back $f^*$
and the push-forward $f_*$ between the corresponding
Stanley-Reisner rings are defined depending on not only $f$ but
$\V$ and $\Vpr$.

The main theorem of the present paper can be stated
in the following form.

\begin{zerotheo}
Let $(\Delta,\V,\xi)$ be a triple of simplicial multi-fan, a set
of edge vectors and a $\Q$-divisor. Let $f:\delpr\to \Delta$
be a birational morphism and $\Vpr$ a set of edge vectors for
the multi-fan $\delpr$. Then
\begin{equation*}\label{eq:pushEll3}
 f_*\hatEst(\delpr,\Vpr,f^*(\xi))=\hatEst(\Delta,\V,\xi).
\end{equation*}
\end{zerotheo}
\begin{zerocoro}
Under the above situation
\begin{equation}\label{eq:pushEll4}
 \hatvarst(\delpr,\Vpr,f^*(\xi))=\hatvarst(\Delta,\V,\xi).
\end{equation}
\end{zerocoro}

The canonical class $K_X$ of a $\Q$-factorial toric variety $X$
corresponds to $K_\Delta=-\sum_ix_i$ where $\Delta$ is the fan
associated to $X$ and $x_i$ is the generator of
the Stanley-Reisner ring corresponding to $D_i$.
For a divisor $D=\sum_ia_iD_i$ we put $\xi=-K_\Delta-\sum_ia_ix_i$.
Then the singular
elliptic genus $Ell_{sing}(X,D)$ is equal to $\hatvarst(\Delta,\V,\xi)$
up to a multiplicative constant depending only on the dimension of $X$
where all the vectors in $\V$ are taken primitive;
cf. Remark \ref{rema:singEll}.

As an example we take a complete non-singular toric variety (more generally a
torus manifold) $X$. Let $G$ be a finite subgroup of the (compact) torus $T$
acting on $X$. Let $a_i$ be the order of the isotropy subgroup of $G$ at
a generic point in $D_i$. Then the ramification divisor of the quotient
map $X\to X/G$ is $\sum_i(a_i-1)D_i$. Let $\Delta$ be the fan associated to
the toric variety $X/G$, and set $\delpr=\Delta$,
$f=id: \delpr \to \Delta$, $\V=\{v_i\}$
with all the $v_i$ primitive and $\Vpr=\{a_iv_i\}$. We then
require $\xi$ to satisfy $f^*(\xi)=-K_{\Delta}=\sum_ix_i$. We have
$\xi=\sum_i\frac1{a_i}x_i$\,; see Remark \ref{rema:xiquot}.
In this case \eqref{eq:pushEll4}
is equivalent to \eqref{eq:Ellquot} with $D=0$ and $D_{X/G}$
given by \eqref{eq:ramif}. Note that $\Vpr$
corresponds to the orbifold structure of $X/G$ which has $X$ itself
as an orbifold chart but $\V$ does not in general.

For the proof of Theorem we first prove Corollary. For that purpose
we use an expression of equivariant orbifold elliptic genus as a
character of
the accompanying torus. That character formula was originally due to
Borisov-Libgober in the case of Gorenstein toric varieties \cite{BL1}
and then was generalized to the case of multi-fans by \cite{HM2}.
The formula behaves well with respect to birational morphisms and
reduces the invariance property of orbifold elliptic
genus to its local version.
The local invariance property is derived from
the rigidity-vanishing property of orbifold elliptic genus which
was exploited in \cite{Hat1} and \cite{Hat2}. Theorem itself follows
from the functorial property of push-forward and the local version
of the genus by using Mayer-Vietoris argument.

The paper is organized as follows. In Section 2 materials
concerning multi-fans which are needed later are given. Orbifold
elliptic class and orbifold elliptic genus of simplicial
multi-fans are introduced in Section 3. The main theorem of this
section is Theorem \ref{theo:laurent}. It states that the
(equivariant) orbifold elliptic genus is a character of the torus.
In the case of multi-fans associated with torus orbifolds the
orbifold elliptic genus is the index of a Dirac operator with
bundle coefficients, from which Theorem follows. For general
simplicial multi-fans a combinatorial proof is needed. In Section
4 vanishing theorems are given generalizing those of \cite{Hat2}.
Section 5 is devoted to a character formula of the
Borisov-Libgober type generalizing that of \cite{HM2}. The
invariance property of orbifold elliptic genus (Corollary above)
is stated in Section 6 as Theorem \ref{thm:crepant1}. Its proof is
given in Section 7 after the local version (Theorem \ref{thm:Bn})
is established. The functorial property of orbifold elliptic class
(Theorem above) is stated as Theorem \ref{thm:rhoclass} and is
proved in Section 8. The final section is devoted to a
generalization of orbifold elliptic genus to not necessarily
simplicial multi-fans. It will be shown that the orbifold elliptic
genus can be defined for triples $(\Delta,\V,\xi)$ with $\xi$
$\Q$-Cartier and, in particular, for $\Q$-Gorenstein pairs
$(\Delta,\V)$.

\section{Preliminaries}\label{sec:2}

We shall recall basic definitions and facts concerning simplicial
multi-fans which will be used in the sequel.
For details we refer to \cite{HM1}, \cite{HM2}.

Let $L$ be a lattice of rank n (the notation $N$ is customary in
literatures instead of $L$). A \emph{simplicial multi-fan} in $L$
is a triple $\Delta=(\Sigma,C,w^{\pm})$. Here $\Sigma$ is an
augmented finite simplicial set, that is, $\Sigma$ is a finite
simplicial set with empty set $*=\emptyset$ added as the unique
$(-1)$-dimensional simplex. $\sigmk$ denotes the $k-1$ skeleton of
$\Sigma$ so that $*\in \Sigma^{(0)}$. We assume that
$\Sigma=\coprod_{k=0}^n\sigmk$, and $\sigmn \not=\emptyset$. We
further assume that any $J\in \Sigma$ is contained in some $I\in
\sigmn$.

$C$ is a map from $\sigmk$ into the set of $k$-dimensional
strongly convex rational
polyhedral cones in the vector space $L_\R=L\otimes \R$ for
each $k$ such that $C(*)=\{0\}$, and if $J$ is a face
of $I$, then $C(J)$ is a face of $C(I)$. We require the
following condition is satisfied.
For any $I\in\Sigma$ the map $C$ restricted on
$\{J\in \Sigma\mid J\subset I\}$ is an isomorphism of ordered sets
onto the set of faces of $C(I)$.

$w^{\pm}$ are maps $\sigmn \to \Z_{\ge 0}$ which, when $\Sigma$
is complete, satisfy certain compatibility conditions, as
we shall explain below.
We set $w(I)=w^+(I)-w^-(I)$.

For each $K\in \Sigma$ we set
\[ \Sigma_K=\{ J\in\Sigma\mid K\subset J\}.\]
It inherits the partial ordering from $\Sigma$ and becomes an augmented
simplicial set where $K$ is the unique
minimum element in $\Sigma_K$.
Let $(L_K)_\R$ be the linear subspace of $L_\R$ generated by $C(K)$.
Put $L_K=L\cap (L_K)_\R$ and define
$L^K$ to be the quotient lattice of $L$ by $L_K$.
For $J\in \Sigma_K$ we define $C_K(J)$ to be
the cone $C(J)$ projected on $L^K\otimes\R$.
We define two functions
\[ {w_K}^\pm\colon  \Sigma_K^{(n-|K|)}\subset \Sigma^{(n)} \to \Z_{\ge 0}\]
to be the restrictions of $w^\pm$ to $\Sigma_K^{(n-|K|)}$. The triple
$\Delta_K:=(\Sigma_K,C_K,{w_K}^\pm)$ is a multi-fan in $L^K$
and is called the
\emph{projected multi-fan} with respect to $K\in \Sigma$.
For $K=\emptyset$, the projected multi-fan $\Delta_K$
is nothing but $\Delta$.

A vector $v\in L_\R$ will be called \emph{generic} if $v$ does not
lie on any linear subspace spanned by a cone in $C(\Sigma)$ of
dimension less than $n$. For a generic vector $v$ we set
$d_v=\sum_{v\in C(I)}w(I)$, where the sum is understood to be zero
if there is no such $I$.

\begin{defi}
A simplicial multi-fan $\Delta=(\Sigma,C,w^\pm)$ is called
\emph{pre-complete} if the integer $d_v$ is independent of
generic vectors $v$. In this case this integer will be called
the {\it degree} of $\Delta$ and will be denoted by $\deg(\Delta)$.
A pre-complete multi-fan $\Delta$ is said to be \emph{complete} if the
projected multi-fan $\Delta_K$ is pre-complete for any $K\in \Sigma$.
\end{defi}

A multi-fan is complete if and only if the projected multi-fan
$\Delta_J$ is pre-complete for any $J\in \Sigma^{(n-1)}$.
Let $v$ be a vector such that
its projection $\bar{v}$ is generic for the multi-fan
$\Delta_K$. For a complete multi-fan we have
\begin{equation*}\label{eq:degK}
 \deg(\Delta_K)=\sum_{I\in S_v(K)}w(I)\quad \text{where}\quad
S_v(K)=\{I\in \Sigma_K^{(n-k)}\mid \bar{v}\in C_K(I)\}.
\end{equation*}

In the sequel we shall often consider a set $\V$ consisting
of non-zero edge vectors $v_i$ for each $i\in \sigmone$ such that
$v_i\in L\cap C(i)$. We do not require $v_i$ to be primitive.
For any $J\in \Sigma$ let $L_{J,\V}$ be the sublattice of $L_J$ generated by
$\{v_i\}_{i\in J}$ and $L_{J,\V}^*$ the dual lattice.
Let $\{u_i^J\}_{i\in J}$ be the basis of $L_{J,\V}^*$ dual to
$\{v_i\}_{i\in J}$. For $I\in \sigmn$ we put
$I(v)=\{i\in I\mid \l \u,v\r<0\}$.
Then it can be easily seen that $S_v(K)$ is written as
\begin{equation}\label{eq:Sv}
S_v(K)=\{I\in \sigmn\mid I(v)\subset K\subset I\}.
\end{equation}

Let $\Delta=(\Sigma,C,w^\pm)$ be a simplicial multi-fan
in a lattice $L$ and $\V=\{v_i\}_{i\in \sigmone}$ a set of prescribed
edge vectors as before. We denote the torus $L_\R/L$ by $T$.
We define the equivariant cohomology $H_T^*(\Delta)$ of a
multi-fan $\Delta$ as the Stanley-Reisner ring of the simplicial
complex $\Sigma$.
Namely let $\{x_i\}$ be indeterminates indexed by $\sigmone$, and
let $R$ be the polynomial ring over the integers generated
by $\{x_i\}$. We denote by $\mathcal{I}$ the ideal in $R$
generated by monomials
$\prod_{i\in J} x_i$ such
that $J\notin \Sigma$. $H_T^*(\Delta)$ is by definition
the quotient $R/\mathcal{I}$.

The dual lattice $L^*$ is canonically identified with
$H^2(BT)$ where $BT$ is a classifying space of $T$ so that
$L^*\otimes \Q$ is identified with $H^2(BT)\otimes\Q$.
We put $L_\V=\bigcap_{I\in \sigmn}L_{I,\V}$. It is a sublattice
of $L$ of the same rank $n$ and $L_\V^*$ contains $L^*$.
$L_\V^*\otimes \Q$ is identified with $L^*\otimes\Q=H^2(BT)\otimes\Q$.
Let $S(L^*)$ be the symmetric algebra of $L^*$. It is identified with
the polynomial algebra $S^*(L)$ over $L$. It is also identified
with $H^*(BT)$. Similarly $S(L^*)\otimes\Q$ and $S(L_\V^*)\otimes\Q$
are identified with $H^*(BT)\otimes\Q$.

We regard $L_\V^*$ as a submodule of $H_T^2(\Delta)$ by the formula
\begin{equation*}\label{eq:structure}
 u=\sum_{i\in \sigmone}\langle u,v_i\rangle x_i .
\end{equation*}
This determines an $S^*(L_\V)$-module structure
of $H_T^*(\Delta)$ and an $H^*(BT)\otimes\Q$-module structure of
$H_T^*(\Delta)\otimes\Q$.
It should be noticed that these module structures depend on
the choice of vectors ${\V}$ as above. In order to emphasize this
dependence we shall write $H_T^*(\Delta,\V)$ instead of $H_T^*(\Delta)$.

For $K\in \sigmk$ let $\{u_i^K\}_{i\in K}$ be the basis of $L_{K,\V}^*$
dual to the basis $\{v_i\}_{i\in K}$ of $L_{K,\V}$ as before.
The restriction homomorphism
$\iota_K^* :H_T^*(\Delta,\V)\to S^*(L_{K,\V})$ is determined by
\begin{equation*}\label{eq:restrict}
 \iota_K^*(x_i)=\begin{cases}
            u_i^K & \text{for}\ i\in K\\
            0     & \text{for}\ i\notin K.
            \end{cases}
\end{equation*}
For $J\subset K$
let $\iota_J^{K*}:S^*(L_{K,\V})\to S^*(L_{J,\V})$ be the induced
homomorphism of the inclusion $\iota_J^K:J\to K$.
If $J\subset K$ then $\iota_J^*=\iota_J^{K*}\circ\iota_K^*$.

For $I\in \sigmn$ we have
\[ \iota_I^*(u)=u \quad \text{for $u\in L_{I,\V}^*$} . \]
In particular $\iota_I^*| L_\V^*$ is
the identity map for any $I\in \sigmn$, and $\iota_I^*$ is an
$S^*(L_\V)$-module
map. Note that $\bigoplus_{I\in \sigmn}\iota_I^*$
embeds $H_T^*(\Delta,\V)\otimes\Q$ into $(H^*(BT)\otimes\Q)^{\sigmn}$.
Its image is given by the subset
consisting of the elements $(u(I))_{I\in \sigmn}$ satisfying
\begin{equation}\label{eq:pistarimage}
 \iota_{I_1\cap I_2}^{I_1*}(u(I_1))
 =\iota_{I_1\cap I_2}^{I_2*}(u(I_2))\ \text{for any $I_1,I_2\in \sigmn$
 with $I_1\cap I_2\not=\emptyset$}.
\end{equation}

Let $S$ be the multiplicative subset of $S^*(L_\V)$ consisting of
non-zero elements in $S^*(L_\V)$ and let $S^{-1}$ denote the
localization by $S$. For $K\in \sigmk$ put
\begin{equation*}\label{eq:HI}
 H_{K,\V}=L_K/L_{K,\V}.
\end{equation*}
$H_{K,\V}$ will be simply written $H_K$ when it is clear that $\V$ is
understood in the context. We define the push-forward
$\pi_*:H_T^{*}(\Delta,\V)\otimes\Q\to S^{-1}H^*(BT)\otimes\Q$ by
\begin{equation*}\label{eq:pistar}
 \pi_*(x)=\sum_{I\in \sigmn}\frac{w(I)i_I^*(x)}{|H_{I,\V}|
\prod_{i\in I}\u}.
\end{equation*}
It is an $H^*(BT)\otimes\Q$-module homomorphism. When $\Delta$ is complete
it is known that the image of $\pi_*$ is contained in $H^*(BT)\otimes\Q)$,
cf. \cite{HM1}.

\begin{rema}
For details concerning torus orbifolds and their associated multi-fans
we refer to \cite{Hat1} and \cite{Hat2}. To a torus
orbifold $X$ a complete simplicial multi-fan $\Delta$ and a set of
edge vectors $\V$ are associated, and there is a canonical map
\[ \kappa: H_T^*(X)\otimes\Q \to H_T^*(\Delta,\V)\otimes\Q \]
which is an $H^*(BT)\otimes\Q$-homomorphism. It becomes an isomorphism
in favorable cases. To each $i\in \sigmone$ there corresponds a
$T$-invariant suborbifold of codimension two whose equivariant
Poincar\'{e} dual is mapped to $x_i\in H_T^2(\Delta,\V)\otimes\Q$ by
$\kappa$. The first Chern class $c_1(X)\in H_T^2(X)\otimes\Q$ is mapped
to $\sum_{i\in \sigmone}x_i$, cf. \cite{Hat1}, Remark 3.2 and
\cite{Hat2}, Remark 2.5.
\end{rema}

\begin{rema}
Suppose that $X$ is a $\Q$-factorial toric variety and $\Delta$ is
its associated fan. In this case one usually takes all the
vectors in $\V=\{v_i\}$ primitive. Then $H_T^2(\Delta,\V)$ is
identified with the module of all $T$-Weil divisors.
Moreover $L^*=H^2(BT)$ is contained in the submodule $Div_TX$ of the
$T$-Cartier divisors, and the quotients $Div_TX/L^*$ and
$H_T^2(\Delta,\V)/L^*$ are identified with the Picard group
$Pic(X)$ and the divisor class group $A_{n-1}(X)$ respectively.
The element $-\sum_{i\in \sigmone}x_i$ corresponds to a canonical
divisor $K_X$. See \cite{Ful} Sections 3.4, 4.3 and 4.4.
\end{rema}

In this paper elements of $H_T^2(\Delta,\V)$ and
$H_T^2(\Delta,\V)\otimes\Q$ will be called \emph{divisors} and
$\Q$-\emph{divisors} respectively. In the same spirit we adopt the
following
\begin{defi}
A divisor $\xi$ is called $T$-\emph{Cartier} if $\iota_I^*(\xi)$
is contained in $L^*=H^2(BT)$ for all $I\in \sigmn$.
\end{defi}

We need another description of the group $H_K=H_{K,\V}$. For simplicity
identify the set $\sigmone$ with $\{1,2,\ldots ,m\}$ and
consider a homomorphism $\eta:\R^m=\R^{\sigmone} \to L_\R$ sending
$\mathbf{a}=(a_1,a_2,\ldots,a_m)$ to $\sum_{i\in \sigmone}a_iv_i$.
For $K\in\sigmk$ we define
\[ \tilG_{K,\V}=\{\mathbf{a} \mid \eta(\mathbf{a})\in L\ \text{and}\
a_j=0\ \text{for $j\not\in K$}\} \]
and define $G_{K,\V}$ to be the image of $\tilG_{K,\V}$ in
$\tilT=\R^m/\Z^m$.
It will be written $G_K$ for simplicity.
The homomorphism $\eta$ restricted on $\tilG_{K,\V}$ induces
an isomorphism
\[\eta_K: G_K\cong H_K=H_{K,\V}\subset T=L_\R/L. \]

Put
\[ G_{\Delta}=\bigcup_{I\in\sigmn}G_I\subset \tilT \quad \text{and}\quad
  DG_{\Delta}=\bigcup_{I\in\sigmn}G_I\times G_I\subset
   G_{\Delta}\times G_{\Delta}. \]

Let $v(g)=\mathbf{a}=(a_1,a_2,\ldots,a_m)\in \R^m$ be
a representative of $g\in \tilT$. The factor $a_i$ will be denoted by
$v_i(g)$. It is determined modulo integers. If $g\in G_I$, then
$v_i(g)$ is necessarily a rational number.

Let $g\in G_I$ and $h=\eta_I(g)\in H_I$. Then
$\eta(v(g))\in L_{I}$ is a representative of $h$ in $L_I$ which
will be denoted by $v(h)$. Then
\begin{equation*}\label{eq:vig}
v_i(g)=\l \u, v(h)\r\quad \text{for $i\in I$}.
\end{equation*}
Define a homomorphism $\chi_i:\tilT\to \C^*$ by
\[ \chi_i(g)=e^{2\pi\img v_i(g)}=e^{2\pi\img \l\u, v(h)\r} \quad
\text{for $g\in G_I$ and $i\in I$}. \]
This will also be written $\chi_I(\u,h)$.
It gives a character of $H_I$ for each $i\in I$.

\section{Orbifold elliptic class and orbifold elliptic genus}\label{sec:3}

Let $\Delta$ be a simplicial multi-fan in a lattice $L$
and $\V=\{v_i\}_{i\in\sigmone}$
a set of prescribed vectors as in Section 2.
Let also $\xi=\sum_{i\in \sigmone}d_ix_i\in H_T^2(\Delta,\V)\otimes \Q$
be a $\Q$-divisor.

We shall define the orbifold elliptic class $\hatEst(\Delta,\V,\xi)$ in
$(H_T^{**}(\Delta.\V)\otimes \C)[[q]]$ and
the orbifold elliptic genus $\hatvarst(\Delta,\V,\xi)$ of
the triple $(\Delta,\V,\xi)$.
The definition of
$\hatvarst(\Delta,\V,\xi)$ is
such that, if $X$ is a complete $\Q$-factorial toric variety
and $\Delta(X)$ is its fan,
then $\hatvarst(\Delta(X),\V,\xi_0)$ coincides with
the orbifold elliptic genus of $X$ regarded as an orbifold,
where all the vectors $v_i$ in $\V$ are taken primitive and
$\xi_0=\sum_{i\in\sigmone}x_i$.

We first consider the function $\Phi(z,\tau)$ of
$z$ in $\C$ and $\tau$ in the upper half plane $\mathcal{H}$
defined by the following formula.
\[ \Phi(z,\tau)=(t^{\frac12}-t^{-\frac12})\prod_{k=1}^{\infty}
 \frac{(1-tq^k)(1-t^{-1}q^k)}
      {(1-q^k)^2},  \]
where $t=e^{2\pi\img z}$ and $q=e^{2\pi\img\tau}$. It is meromorphic
with respect to $(z,\tau)$. Note that $|q|<1$.
Let $A=\begin{pmatrix}a & b\\
                      c & d \end{pmatrix} \in SL_2(\Z)$, and
put $A(z,\tau)=(\frac{z}{c\tau +d},\frac{a\tau +b}{c\tau +d})$.
$\Phi$ is a Jacobi form and satisfies the following transformation
formulae, cf. \cite{HBJ}.
\begin{align}
 \Phi(A(z,\tau))&=(c\tau +d)^{-1}
  e^{\frac{\pi\img cz^2}{c\tau +d}}\Phi(z,\tau),
  \label{eq:Transformula1}\\
 \Phi(z+m\tau +n,\tau)&=(-1)^{m+n}e^{-\pi\img(2mz+m^2\tau)}\Phi(z,\tau)
 \label{eq:Transformula2}
\end{align}
where $m,n\in \Z$.

For $\sigma\in \C$ we set
\[ \begin{split}
\phist(z,\tau,\sigma)&=-\frac{\Phi(z+\sigma,\tau)}
 {\Phi(z,\tau)\Phi(\sigma,\tau)} \\
&=\frac{1-\zeta t}{(1-\zeta)(1-t)}
\prod_{k=1}^{\infty}\frac{(1-q^k)^2(1-\zeta tq^k)(1-\zeta^{-1}t^{-1}q^k)}
 {(1-\zeta q^k)(1-\zeta^{-1}q^k) (1-tq^k)(1-t^{-1}q^k)},
\end{split} \]
where $\zeta =e^{2\pi\img\sigma}$.
From \eqref{eq:Transformula1} and \eqref{eq:Transformula2}
the following transformation formulae
for $\phist$ follow:
\begin{align}
 \phist(A(z,\tau),\sigma)&=
  e^{2\pi\img cz\sigma}\phist(z,\tau,(c\tau +d)\sigma),
  \label{eq:transformula1}\\
  \phist(z+m\tau +n,\tau,\sigma)&=e^{-2\pi\img m\sigma}\phist(z,\tau,\sigma)=
  \zeta^{-m}\phist(z,\tau,\sigma).
  \label{eq:transformula2}
\end{align}

\begin{defi}
We define the (equivariant, stabilized) \emph{orbifold elliptic
class}\linebreak $\hatEst(\Delta,\V,\xi)$ of a triple
$(\Delta,\V,\xi)$ by
\begin{equation}\label{eq:ellipticclass}
 \hatEst(\Delta,\V,\xi)=\sum_{(g_1,g_2)\in DG_{\Delta}}\prod_{i\in \sigmone}
 x_i\zeta^{d_iv_i(g_1)}\phist(-\frac{x_i}{2\pi\img}+v_i(g_1)\tau-
 v_i(g_2),\tau,d_i\sigma).
\end{equation}
\end{defi}

\begin{rema}\label{rema:singEll}
Let $X$ be a complete $\Q$-factorial toric variety and $\Delta$ the
fan associated to $X$. Put $K_\Delta=-\sum_{i\in\sigmone}x_i$. For
$\xi=\sum_id_ix_i$ define
$D_\Delta=\sum_ia_ix_i\in H^2(\Delta,\V)\otimes\Q$ by
$\xi=-K_\Delta-D_\Delta$, that is $a_i=1-d_i$. For the divisor
$D=\sum_ia_iD_i$ in $X$ the singular
elliptic class $\mathcal{E}\ell\ell_{sing}(X,D)$ defined by
Borisov and Libgober is equal to
$(-\Psi(\sigma,\tau))^{\dim X}\hatEst(\Delta,\V,\xi)$ where all vectors
$v_i$ in $\V$ are taken primitive.
\end{rema}

The right hand side of \eqref{eq:ellipticclass}
does not depend on the choice of representatives
$v(g_1),\ v(g_2)$ in view of \eqref{eq:transformula1} and
\eqref{eq:transformula2}.
It is sometimes useful to take a representative $v(g)$ of
$g\in G_\Delta$ such that
\begin{equation}\label{eq:fgi}
 0\leq v_i(g) <1 \ \text{for all $i\in I$}.
\end{equation}
Such a representative is unique. We denote the value
$v_i(g)$ by $f_{g,i}$ for such a representative
$v(g)$. The sum $\sum_{i\in \sigmone}d_if_{g,i}$ will be
denoted by $f_g(\xi)$.

For $h=\eta_I(g)\in H_I$ we put $f_{h,i}=f_{g,i}$. It is equal to
$\l\u,v(h)\r$ for a uniquely determined representative $v(h)$ of
$h$. If $h$ lies
in $H_K$ for $K\in \sigmk$ contained in $I$,
then $f_{h,i}=0$ for $i\not\in K$, and $f_{h,i}$
depends only on $K$.

Note that
\begin{equation*}\label{eq:est}
\small
 \begin{split}
\phist&(-\frac{x_i}{2\pi\img}
 +v_i(g_1)\tau-v_i(g_2),\tau,d_i\sigma) \\
 & =\frac{1}{1-\zeta^{d_i}}\prod_{k=1}^\infty
\frac{(1-q^k)^2}{(1-\zeta^{d_i} q^k)(1-\zeta^{-d_i}q^k)}\cdot
\frac{1-\zeta^{d_i}\xi_i}{1-\xi_i}
\prod_{k=1}^{\infty}\frac{(1-\zeta^{d_i}\xi_iq^k)
 (1-\zeta^{-d_i}\xi_i^{-1}q^k)}
 {(1-\xi_iq^k)(1-\xi_i^{-1}q^k)}, \end{split}
\end{equation*}
where
\[  \xi_i=e^{-x_i}q^{v_i(g_1)}e^{-2\pi\img v_i(g_2)}=
 e^{-x_i}q^{v_i(g_1)}\chi_i(g_2)^{-1}. \]
$e^{-x_i}$ is considered as an element of the completed ring
$H_T^{**}(\Delta,\V)\otimes\Q$. Let $r$ be the least common multiple
of $\{\vert H_I\vert\}_{I\in \sigmn}$.
The right-hand side of \eqref{eq:ellipticclass} defines
an element in $(H_T^{**}(\Delta,\V)\otimes\C)[[q^{\frac1{r}}]]$.
It will be shown in Theorem \ref{theo:laurent} that it in fact lies in
$(H_T^{**}(\Delta,\V)\otimes\C)[[q]]$.

\begin{defi}
The (equivariant, stabilized) \emph{orbifold elliptic genus}
$\hatvarepst(\Delta,\V,\xi)\in (H^{**}(BT)\otimes\C)[[q]]$ is
defined as the image $\pi_*(\hatEst(\Delta,\V,\xi)$.
\end{defi}
Explicitly we have
\begin{prop}\label{prop:est2}
\begin{equation}\label{eq:est2}
 \small
 \begin{split}
 \hatvarepst&(\Delta,\V,\xi) \\
 &= \sum_{I\in\sigmn}
   \frac{w(I)}{|H_I|}
   \sum_{h_1\in H_I\atop h_2\in H_I}
  \prod_{i\in I}\zeta^{d_i\l u_i^I, v(h_1)\r}
 \phist\Bigl(-\frac{u_i^I}{2\pi\img}+ \l\u, v(h_1)\tau-v(h_2)\r, \tau, d_i\sigma\Bigr).
\end{split}
\end{equation}
\end{prop}
\begin{proof}
We take a representative $v(g)$ for each $g\in G_\Delta$ satisfying
\eqref{eq:fgi}.
Let $(g_1,g_2)\in DG_\Delta$ and $I\in \sigmn$. Note that
$x_j\phist(-\frac{x_j}{2\pi\img}+v_j(g_1)\tau-v_j(g_2),\tau,d_j\sigma)$
is of the form
\begin{equation*}\label{eq:factor}
 \frac{x_j}{1-\xi_j}\phi_j(x_j).
\end{equation*}
If $g_1\not\in G_I$ or $g_2\not\in G_I$,
then there is a $j\not\in I$ such that $v_j(g_1)\not=0$ or
$v_j(g_2)\not=0$. For such $j$ the Maclaurin expansion with
respect to $x_j$ of the factor $\frac{x_j}{1-\xi_j}$ has no constant term.
Noting that $\iota_I^*(x_j)=0$ we have $\iota_I^*(\frac{x_j}{1-\xi_j})=0$
and hence
\begin{equation*}\label{eq:phist}
 \iota_I^*(x_j\zeta^{d_jv_j(g_1)}\phist(-\frac{x_j}{2\pi\img}+v_j(g_1)\tau-
 v_j(g_2),\tau,d_j\sigma))=0.
\end{equation*}
Therefore only elements $(g_1,g_2)$ in $G_I\times G_I$
contribute to $\iota_I^*(\hatEst(\Delta,\V,\xi))$.

Next suppose $(g_1,g_2)$ lies in $G_I\times G_I$. If $j\not\in I$, then
$v_j(g_1)=v_j(g_2)=0$ and $\iota_I^*(x_j)=0$. In particular
$\iota_I^*(\xi_j)=1$ and $\iota_I^*(\frac{x_j}{1-\xi_j})=1,
\ \iota_I^*(\phi_j(x_j))=1$. Hence
\[ \iota_I^*(x_j\zeta^{d_jv_j(g_1)}\phist(-\frac{x_j}{2\pi\img}+
 v_j(g_1)\tau- v_j(g_2),\tau,d_j\sigma))=1. \]
It follows that
\begin{equation}\label{eq:iotaE}
\small
 \begin{split}
 \iota_I^*(\hatEst&(\Delta,\V,\xi)) \\
 &=\iota_I^*\left(\sum_{(g_1,g_2)\in G_I\times G_I}
 \prod_{i\in I}x_i\zeta^{d_iv_i(g_1)}
 \phist(-\frac{x_i}{2\pi\img}+v_i(g_1)\tau
 -v_i(g_2),\tau,d_i\sigma)\right) \\
 &=\!\!\!\!\sum_{(h_1,h_2)\in H_I\times H_I}
 \prod_{i\in I}\u\zeta^{d_i\l\u, v(h_1)\r}
 \phist\Bigl(-\frac{\u}{2\pi\img}+\l\u, v(h_1)\r\tau
 -\l\u, v(h_2)\r,\tau,d_i\sigma\Bigr).
\end{split}
\end{equation}
Here $v(h)$ denotes a representative of $h\in H_I$ such that
$0\leq \l\u, v(h)\r<1$. Then we have
\begin{equation*}\label{eq:hatvar2}
\small
 \begin{split}\pi_*(&\hatEst(\Delta,\V,\xi)) \\
 &=\!\!\!
 \sum_{I\in \sigmn}\frac{w(I)}{|H_I|}
 \sum_{h_1\in H_I\atop h_2\in H_I}
 \prod_{i\in I}\zeta^{d_i\l\u, v(h_1)\r}
 \phist\Bigl(-\frac{\u}{2\pi\img}+\l\u, v(h_1)\r\tau-
 \l\u, v(h_2)\r,\tau,d_i\sigma\Bigr).
 \end{split}
\end{equation*}
\end{proof}

\begin{note}
The right hand side of \eqref{eq:est2} is independent of the
choice of representatives $v(h_1)$ and $v(h_2)$ though we
used representatives satisfying \eqref{eq:fgi} in the proof above.
\end{note}

Put $L_\C=L\otimes \C$. We shall also consider a function
$\hatvarst(\Delta,\V,\xi)$ from $L_\C$ into $\C[[q]]$
defined by
\begin{equation}\label{eq:hatvar}
\small
 \begin{split}
 \hatvarst&(\Delta,\V,\xi)(w) \\
 &= \sum_{I\in\sigmn}
   \frac{w(I)}{|H_I|}
   \sum_{(h_1,h_2)\in H_I\times H_I}
  \prod_{i\in I}\zeta^{d_i\l u_i^I, v(h_1)\r}
 \phist\Bigl(\l u_i^I, -w+\tau v(h_1)-v(h_2)\r, \tau, d_i\sigma\Bigr).
\end{split}
\end{equation}
This function $\hatvarst(\Delta,\V,\xi)$ is also called
(stabilized) \emph{orbifold elliptic genus}. Later it will be
shown that $\hatvarst(\Delta,\V,\xi)$ belongs to
$(R(T)\otimes \C)[[q]]$ and $ch(\hatvarst(\Delta,\V,\xi))=
\hatvarepst(\Delta,\V,\xi)$. Here $R(T)$ is the character ring of
the torus $T$ and its elements are considered as functions on
$L_\C$ via the projection $L_\C\to T_\C$ where $T_\C$ is
the complexification of $T$. For $\xi_0=\sum_{i\in \sigmone}x_i$
the genus $\hatvarst(\Delta,\V,\xi_0)$ was introduced in \cite{HM2}
in a non-stabilized form as the orbifold elliptic genus of the
pair $(\Delta,\V)$ and was denoted by $\hatvar(\Delta,\V)$,
cf. also \cite{Hat2}.
$\hatvarst(\Delta,\V,\xi)$ is sometimes written as
$\hatvarst(\Delta,\V,\xi;w,\tau,\sigma)$ to
emphasize the variables.
Also $\tau$ and $\sigma$ are often considered as parameters; in this
case it is sometimes written as $\hatvarst(\Delta,\V,\xi;\tau,\sigma)$ to
emphasize the parameters.

A vector $v\in L$ can be considered as a homomorphism
$\C\ni z\mapsto zv\in L_\C$. Let $v^*(\hatvarst(\Delta,\V,\xi))$ be
the pull-back induced by $v$. It will be denoted by
\[ \hatvarvst(\Delta,\V,\xi), \]
and will be called the (stabilized)
\emph{orbifold elliptic genus along} $v$
of the triple $(\Delta,\V,\xi)$. Explicitly it is given by
\begin{equation}\label{eq:hatvarv}
\small
 \begin{split}
 \hatvarvst&(\Delta,\V,\xi)(z) \\
 &= \sum_{I\in\sigmn}
   \frac{w(I)}{|H_I|}
   \sum_{(h_1,h_2)\in H_I\times H_I}
  \prod_{i\in I}\zeta^{d_i\l u_i^I, v(h_1)\r}
 \phist\Bigl(\l u_i^I, -zv+\tau v(h_1)-v(h_2)\r, \tau, d_i\sigma\Bigr).
\end{split}
\end{equation}

If $J\subset K$ then we have
$ L_J\cap L_{K,\V}=L_{J,\V} $, and hence
$H_J$ is canonically
embedded in $H_K$. We set
\[\hat{H}_K=H_K\setminus \bigcup_{J\subsetneqq K}H_J .\]
The subset $\hat{H}_K$ is characterized by
\begin{equation}\label{eq:0}
 \hat{H}_K=\{h\in H_K\mid \langle u_i^K,v(h)\rangle \not\in \Z
\quad \text{for any $i\in K$} \},
\end{equation}
where $v(h)\in L_K$ is a representative of $h\in H_K$.
For the minimum element $*=\emptyset\in \Sigma^{(0)}$
we set $\hat{H}_*=H_*=0$.
Note that \eqref{eq:0} can be rewritten
as
\begin{equation*}\label{eqn:0bis}
 \hat{H}_K=\{h\in H_K\mid f_{h,i}\not=0
\quad \text{for any $i\in K$} \}.
\end{equation*}

If $K$ is contained in $I\in \sigmn$, then the canonical map
$L_{I,\V}^*\to L_{K,\V}^*$ sends $u_i^I$ to $u_i^K$ for $i\in K$
and to $0$ for $i\in I\setminus K$. Therefore, if $h$ is in $H_K$,
then $\l u_i^I,v(h)\r=0$ for $i\in I\setminus K$, and
$\l u_i^I,v(h)\r=\l u_i^K,v(h)\r$ for $i\in K$. Here $v(h)\in L_K$ is
regarded as lying in $L$. This observation
leads to the following expression of $\hatvarst(\Delta,\V,\xi)$
which is sometimes useful.
\begin{equation}\label{eq:hatvarv2}
\small
 \begin{split}
 \hatvarst(&\Delta,\V,\xi)(w)
 = \sum_{k=0}^n\sum_{K\in\sigmk,h_1\in \hat{H}_K}
   \zeta^{\l u^K(\xi), v(h_1)\r}
   \sum_{I\in \Sigma_K^{(n-k)}}\frac{w(I)}{|H_I|}\cdot \\
   &\sum_{h_2\in H_I}
    \prod_{i\in I\setminus K}
    \phist(-\l u_i^I, w+v(h_2)\r, \tau, d_i\sigma)
    \prod_{i\in K}
    \phist(-\l u_i^I, w-\tau v(h_1)+v(h_2)\r, \tau, d_i\sigma),
\end{split}
\end{equation}
where $u^K(\xi)=\iota_K^*(\xi)=\sum_{i\in K}d_iu_i^K$.

\begin{theo}\label{theo:laurent}
$\hatEst(\Delta,\V,\xi)$ belongs to $(H_T^{**}(\Delta,\V)\otimes\C)[[q]]$.
$\hatvarepst(\Delta,\V,\xi)$ belongs to $(H^{**}(BT)\otimes\C)[[q]]$.
Moreover, if $\Delta$ is complete, then
$\hatvarst(\Delta,\V,\xi)$ belongs to $(R(T)\otimes\C)[[q]]$.
\end{theo}
\begin{proof}
Since the map $\bigoplus_{I\in \sigmn}\iota_I^* :
(H^{**}_T(\Delta,\V)\otimes\C)[[q^{\frac1{r}}]] \to
((H^{**}(BT)\otimes\C)[[q^{\frac1{r}}]])^{\sigmn}$ is injective it is
enough to prove that $\iota_I^*(\hatEst(\Delta,\V,\xi))$ belongs to
$(H^{**}(BT)\otimes\C)[[q]]$ for any $I\in \sigmn$.

Fix $I$ and $g_1\in G_I$. In view of \eqref{eq:iotaE} it suffices to
show that
\[ \begin{split}
 &\iota_I^*\left(\sum_{g_2\in G_I}\prod_{i\in I}
 \phist(-\frac{x_i}{2\pi\img}+ \tau v_i(g_1)-v_i(g_2),\tau,d_i\sigma)\right) \\
 =&\sum_{h_2\in H_I}
 \prod_{i\in I}\phist(-\frac{\u}{2\pi\img}+\l\u, v(h_1)\r\tau
 -\l\u, v(h_2)\r,\tau,d_i\sigma)
 \end{split}
\]
belongs to $(H^{**}(BT)\otimes\C)[[q]]$, where $v(g)$ is the
unique representative of $g\in G_I$ satisfying \eqref{eq:fgi}.
We introduce auxiliary variables $\tau_1$ with $\Im(\tau_1)>0$ and
put $q_1=e^{2\pi\img \tau_1}$. We put
\[ \mathcal{E}(q_1)=
 \sum_{g_2\in G_I}\prod_{i\in I}
 \phist(-\frac{x_i}{2\pi\img}+ \tau_1v_i(g_1)-v_i(g_2),\tau,d_i\sigma).
\]
and expand it with respect to $q$:
\[ \mathcal{E}(q_1)= \sum_{s\in \Z_{\geq 0}}\mathcal{E}_s(q_1)q^s. \]

Put $X_i(g_2)=\chi_i(g_2)^{-1}e^{-x_i}q_1^{v_i(g_1)}$. Since
\begin{multline*}
\phist(-\frac{x_i}{2\pi\img}
 +\tau_1v_i(g_1)-v_i(g_2),\tau,\sigma_1) \\
 =\frac{1}{1-\zeta^{d_i}}\prod_{k=1}^\infty
\frac{(1-q^k)^2}{(1-\zeta^{d_i}q^k)(1-\zeta^{-d_i}q^k)}\cdot\\
\frac{1-\zeta^{d_i}X_i(g_2)}{1-X_i(g_2)}
\prod_{k=1}^{\infty}\frac{(1-\zeta^{d_i}X_i(g_2)q^k)
 (1-\zeta^{-d_i}X_i(g_2)^{-1}q^k)}
 {(1-X_i(g_2)q^k)(1-X_i(g_2)^{-1}q^k)},
\end{multline*}
$\mathcal{E}_s(q_1)$ is expanded in a Laurent series
\begin{equation*}\label{eq:est3}
 \mathcal{E}_s(q_1)=\sum_{c=(c_1,\ldots,c_n)}a_c
 \sum_{g_2\in G_I}X_1(g_2)^{c_1}\cdots X_{n}(g_2)^{c_n}
\end{equation*}
where $I$ is indexed by $\{1,\ldots,n\}$.
Moreover the sum of negative exponents $c_i$ is bounded below
by $-s$, i.e., $-s\leq\sum_{c_i<0}c_i$ for all $c$.
If we put $u_c=\sum_{i\in I}c_iu_i^I$, then
\[ \iota_I^*(\sum_{g_2\in G_I}X_1(g_2)^{c_1}\cdots X_{n}(g_2)^{c_n})
 =\sum_{h_2}\chi_I(u_c, h_2)e^{-u_c}q_1^{\l u_c,v(h_1)\r}.  \]
Since $u_c$ lies in $L_{I,\V}^*$, $\chi_I(u_c,\ )$ is a character of
$H_I$, and we have
\[ \sum_{h\in H_I}\chi_I(u_c, h)=\begin{cases}
 |H_I|, & \text{if $u_c\in L_I^*$,} \\
  0,    & \text{if $u_c\not\in L_I^*$.}\end{cases} \]
Hence $\iota_I^*(\mathcal{E}_s(q_1))=|H_I|
\sum_{c:u_c\in L_I^*}a_ce^{-u_c}q_1^{\l u_c,v(h_1)\r}$.
Since ${\l u_c,h_1\r}$ is an integer for $u_c\in L_I^*$,
$\iota_I^*(\mathcal{E}_s(q_1))$ is a Laurent series in $q_1$.
Furthermore, since $0<f_{h_1,i}<1$ and $-s\leq\sum_{c_i<0}c_i$,
we have $-s\leq \sum_ic_if_{h_1,i}=\l u_c,v(h_1)\r$.
Noting that $\mathcal{E}(q)=
\sum_{s\in \Z_{\geq 0}}\mathcal{E}_s(q)q^s$, we see that
$\iota_I^*(\mathcal{E}(q))$ has no
negative powers. Thus $\iota_I^*(\mathcal{E}(q))$ is
a power series in $q$. This proves that
$\hatEst(\Delta,\V,\xi)$ belongs to $(H_T^{**}(\Delta,\V)\otimes\C)[[q]]$.
Since $\hatvarepst(\Delta,\V,\xi)=\pi_*(\hatEst(\Delta,\V,\xi)$ it belongs
to $(H^{**}(BT)\otimes\C)[[q]]$.

To prove the second part fix $K\in\sigmk$ and $h_1\in \hat{H}_K$. In view of
\eqref{eq:hatvarv2} it suffices to show that
$\bar{\varphi}_{h_1}(\Delta,\V,\xi)$ defined by
\begin{equation*}\label{eq:bar}
 \bar{\varphi}_{h_1}(\Delta,\V,\xi)(w)=\sum_{I\in\Sigma_K^{n-k}}
   \frac{w(I)}{|H_I|}\sum_{h_2\in H_I}\prod_{i\in I}
 \phist(\l u_i^I, -w+\tau v(h_1)-v(h_2)\r, \tau, d_i\sigma)
\end{equation*}
is expanded in a formal power series in $q$ with coefficients in
$R(T)\otimes\C$.

 Take $g_1\in G_\Delta$ such that
$\eta_K(g_1)=h_1$. Then $v_i(g_1)=0$ for $i\not\in K$ since
$h_1\in\hat{H}_K$. Put $\Sigma^{'(1)}=\bigcup_{I\supset K}I$ and
consider the following two quantities
\[ \begin{split}
 \bar{\mathcal{E}}(q_1)&=
 \sum_{g_2\in G_\Delta}\prod_{i\in \Sigma^{'(1)}}
 \phist(-\frac{x_i}{2\pi\img}+ \tau_1v_i(g_1)-v_i(g_2),\tau,d_i\sigma), \\
 \bar{\varphi}(q_1)(w)&=
 \sum_{I\in\Sigma_K^{n-k}}
 \frac{w(I)}{|H_I|}\sum_{h_2\in H_I}\prod_{i\in I}
 \phist(\l u_i^I, -w+\tau_1v(h_1)-v(h_2)\r, \tau, d_i\sigma),
\end{split} \]
and expand them with respect to $q$:
\[ \begin{split}
 \bar{\mathcal{E}}(q_1)&=
 \sum_{s\in \Z_{\geq 0}}
 \bar{\mathcal{E}}_s(q_1)q^s, \\
 \bar{\varphi}(q_1)(w)&=
 \sum_{s\in \Z_{\geq 0}}
 \bar{\varphi}_s(w,q_1)q^s.
\end{split} \]
Note that
\[ \bar{\varphi}(q)(w)
=\bar{\varphi}_{h_1}(\Delta,\V,\xi)(w). \]

As in the proof of the first part we see that
$\bar{\mathcal{E}}_s(q_1)$ is
expanded in the following form:
\[
 \bar{\mathcal{E}}_s(q_1)=\sum_{c=(c_1,\ldots,c_{m'})}
 b_c\sum_{g_2\in G_\Delta}
 \frac{\prod_{i\in \Sigma^{'(1)}}
 \chi_i(g_2)^{-c_i}e^{-c_ix_i}q_1^{c_iv_i(g_1)}}
 {\prod_{i\in \Sigma^{'(1)}\setminus K}
 (1-\chi_i(g_2)^{-1}e^{-x_i})}.
\]
If we put $x_c=\sum_{i\in \Sigma^{'(1)}}c_ix_i$, then,
by a similar argument to the proof of Proposition \ref{prop:est2},
we see that $\bar{\varphi}_s(w,q_1)$ is expanded in the form
\begin{equation*}\label{eq:laurent2}
 \bar{\varphi}_s(w,q_1)=\sum_c b_c\!\!
 \sum_{I\in \Sigma_K^{n-k}}\frac{w(I)}{|H_I|}\!\sum_{h_2\in H_I}
 \frac{\chi_I(\iota_I^*(x_c),h_2)^{-1}e^{-2\pi\img\l\iota_I^*(x_c), w\r}
 q_1^{\l\iota_I^*(x_c), v(h_1)\r}}
 {\prod_{i\in I\setminus K}
 (1-\chi_I(\u,h_2)^{-1}e^{-\l\u,w\r})}.
\end{equation*}
Corollary 2.4 of \cite{HM2} says that
\begin{equation*}\label{eq:varcqone}
 \sum_{I\in \Sigma_K^{n-k}}\frac{w(I)}{|H_I|}\sum_{h_2\in H_I}
 \frac{\chi_I(\iota_I^*(x_c),h_2)^{-1}e^{-2\pi\img\l\iota_I^*(x_c), w\r}
 q_1^{\l\iota_I^*(x_c), v(h_1)\r}}
 {\prod_{i\in \Sigma^{'(1)}\setminus K}
 (1-\chi_I(\u,h_2)^{-1}e^{-\l\u,w\r})}
\end{equation*}
belongs to $R(T)\otimes\C[q_1,q_1^{-1}]$ as a function of $w$. Hence
$\bar{\varphi}_s(w,q_1)$ also does so.

Then, by a similar argument to the proof of the first part,
$\bar{\varphi}_{h_1}(\Delta,\V,\xi)(w)=
 \bar{\varphi}(\tau)(w)= \sum_{s\in \Z_{\geq 0}} \bar{\varphi}_s(w,q)q^s$
is expanded in a power series
$\sum_{s\in \Z_{\geq 0}}\hatvar_s(w)q^s$ in $q$ with
$\hatvar_s\in R(T)\otimes\C$.
\end{proof}

\begin{prop}
Assume that $\Delta$ is complete. Then
the Chern character $ch:(R(T)\otimes\C)[[q]]\to H^{**}(BT)\otimes\C[[q]]$
sends $\hatvarst(\Delta,\V,\xi)$ to
$\hat{\varepsilon}_{st}(\Delta,\V,\xi)$.
\end{prop}
\begin{proof}
The element $t^u$ of the character ring $R(T)$ corresponding to $u\in L^*$
is considered as a function on $L_\C$ defined by $e^{2\pi\img\l u,\ \ \r}$.
On the other hand $ch(t^u)=e^u$ where $L^*$ is identified with
$H^2(BT)$. Hence $ch(e^{2\pi\img\l u,\ \ \r})=e^u$ or
$ch(e^{\l u,\ \ \r})=e^{\frac{u}{2\pi\img}}$. Thus, comparing
\eqref{eq:est2} with \eqref{eq:hatvar}, we see that
\begin{equation}\label{eq:chhatvar}
 ch(\hatvarst(\Delta,\V,\xi))=\hat{\varepsilon}_{st}(\Delta,\V,\xi).
\end{equation}
\end{proof}

\begin{rema}\label{rem:chhatvar}
For a not necessarily complete simplicial multi-fan $\Delta$ we can
consider the function $\hatvarst(\Delta,\V,\xi)$ as a formal power
series in $q$ whose coefficients are rational functions with denominators
of the form $\prod_i(1-\alpha_i t^{u_i}),\ u_i\in L_\V^*,\ \alpha_i\not=0$.
The map $ch$ is extended on such rational functions by
$ch(1-\alpha t^u)=1-\alpha e^u\in H^{**}(BT)\otimes\C$. In this sense
\eqref{eq:chhatvar} holds for general simplicial multi-fans.
\end{rema}


\section{Rigidity and vanishing property}
Let $N>1$ be an integer. We shall consider the following condition for
$\xi\in H_T^2(\Delta,\V)\otimes\Q$:
\begin{equation}\label{eq:xiNdivide}
 \xi=N\eta +u,\ \text{for some}\ \eta\in H_T^2(\Delta,\V),
 \ u\in L_\V^*\otimes \Q.
\end{equation}
If the condition \eqref{eq:xiNdivide} is satisfied, then the classes of
$u^I(\xi)=\iota_I^*(\xi)$ and $u$ regarded as elements in the quotient
$L_\V^*\otimes \Q/NL_\V^*$ are the same. Let $v$ be a vector
in $L_\V$. Since
$\l \iota_I^*(\eta),v\r$ is an integer for all $I\in \sigmn$, the values
$\l u^I(\xi),v\r$ and  $\l u,v\r$ regarded as elements in $\Q/N\Z$
are equal. It will be denoted by $h(v)\in \Q/N\Z$.

Note that if $\xi$ satisfies \eqref{eq:xiNdivide} with $\eta$ $T$-Cartier,
then $\l u^I(\xi),v\r\equiv \l u,v\r\ \bmod N\Z$ even for any $v\in L$,
and hence $h(v)\in \Q/N\Z$ is also defined for $v\in L$.

The following two theorems are the main results of this section.
Prototypes of these theorems were already given in \cite{DLM},
\cite{Hat1} and \cite{Hat2}. The proofs given here are in the
same line as those of the cited works.

\begin{theo}\label{theo:hatrigid}
Let $(\Delta,\V,\xi)$ be a triple of complete simplicial multi-fan
in a lattice $L$ of rank $n$, a set of edge vectors and
an element of
$H_T^2(\Delta,\V)\otimes\Q$. Let $N>1$ be an integer. Assume that $\xi$
satisfies \eqref{eq:xiNdivide} with $\eta$ $T$-Cartier.
Then
$\hatvarst(\Delta,\V,\xi;\tau,\sigma)$ with $\sigma=\frac{k}{N},\ 1<k<N,$
is rigid, that is, it is a constant as a function of $w\in L_\C$.
If, moreover, $\xi$ does not belong to $NH_T^2(\Delta,\V)$, then
it constantly vanishes.
\end{theo}

\begin{theo}\label{theo:varrigid}
Let $(\Delta,\V,\xi)$ be a triple of complete simplicial multi-fan
in a lattice $L$ of rank $n$, a set of generating edge vectors and
an element of $H_T^2(\Delta,\V)\otimes\Q$. If
$\xi$ satisfies the equality $\xi=u$ with non-zero
$u\in L_\V^*\otimes \Q$, then
$\hatvarst(\Delta,\V,\xi;\tau,\sigma)$ constantly vanishes.
\end{theo}

The rest of the section is devoted to the proofs of Theorem
\ref{theo:hatrigid} and Theorem \ref{theo:varrigid}.

Let $v\in L_\V$ be a generic vector.
We put $H_I=L/L_{I,\V}$ for $I\in \sigmn$ as before. For $A\in SL_2(\Z)$
we set
\[ (\hatvarvst)^A(\Delta,\V,\xi;z,\tau,\sigma)=
 \hatvarvst(\Delta,\V,\xi;A(z,\tau),\sigma). \]

\begin{lemm}\label{lemm:hatA}
Assume that
$\xi$ satisfies \eqref{eq:xiNdivide}.
Then $(\hatvarvst)^A(\Delta,\V,\xi;z,\tau,\sigma)$ with
$\sigma=\frac{k}{N},\ 0<k<N,$
has the following expression.
\begin{equation}\label{eq:hatA}
\small
 \begin{split}
 (\hatvarvst)^A(&\Delta,\V,\xi;z,\tau,\sigma)=
 e^{-2\pi\img (c\l u,v\r z\sigma)}
 \sum_{I\in \sigmn}\frac{w(I)}{|H_I|}\sum_{(h_1,h_2)\in H_I\times H_I}
 e^{-2\pi\img\l \iota_I^*(\eta), ckzv\r} \\
 &\quad \prod_{i\in I}e^{2\pi\img\l d_iu_i^I,(c\tau+d)\sigma v(h_1)\r}
 \phist(-\l u_i^I, zv-\tau v(h_1)+v(h_2)\r,\tau,(c\tau+d)d_i\sigma).
 \end{split}
\end{equation}
\end{lemm}

\begin{proof}
By definition we have
\begin{equation*}\label{eq:hatA2}
\small
\begin{split}
 (\hatvarvst)&^A(\Delta,\V,\xi;z,\tau,\sigma) \\
 &= \sum_{I\in\sigmn}
   \frac{w(I)}{|H_I|}
   \sum_{h_1\in H_I\atop h_2\in H_I}
  \prod_{i\in I}\zeta^{\l d_iu_i^I, v(h_1)\r}
 \phist\bigl(-\l u_i^I, {\textstyle\frac{zv-(a\tau+b)v(h_1)+(c\tau+d)v(h_2)}{c\tau+d}}\r,
 A\tau, d_i\sigma\bigr).
\end{split}
\end{equation*}

Using \eqref{eq:transformula1}
we get
\begin{equation}\label{eq:brA2}
 \begin{split}
 \prod_{i\in I}\zeta^{\l d_iu_i^I, v(h_1)\r}&
 \phist(-\l u_i^I, \frac{zv-(a\tau+b)v(h_1)+(c\tau+d)v(h_2)}{c\tau+d}\r,
 A\tau, d_i\sigma) \\
 =&\zeta^{\l u^I(\xi),v(h_1)\r}
 e^{2\pi\img\left(c\l u^I(\xi),
 -zv+(a\tau+b)v(h_1)-(c\tau+d)v(h_2)\r \sigma\right)} \\
 &\ \prod_{i\in I}\phist(-\l\u,zv-(a\tau+b)v(h_1)+(c\tau+d)v(h_2)\r,
 \tau,(c\tau+d)d_i\sigma).
 \end{split}
\end{equation}

We have
\[ c\left((a\tau+b)v(h_1)-(c\tau+d)v(h_2)\right)=
 -v(h_1)+(c\tau+d)(av(h_1)-cv(h_2)). \]
Hence
\begin{equation}\label{eq:exp1}
 \begin{split}
 \zeta^{\l u^I(\xi),v(h_1)\r}&
 e^{2\pi\img\left(c\l u^I(\xi),
 -zv+(a\tau+b)v(h_1)-(c\tau+d)v(h_2)\r \sigma\right)} \\
 &=e^{2\pi\img\l u^I(\xi),-zcv\r\sigma}
 e^{2\pi\img\l u^I(\xi),(c\tau+d)(av(h_1)-cv(h_2))\r\sigma}.
 \end{split}
\end{equation}

Since $\xi$ satisfies \eqref{eq:xiNdivide}, we get
\begin{equation}\label{eq:4-2}
 e^{2\pi\img\l u^I(\xi), zv\r\sigma}=
 e^{2\pi\img\l u,zv\r\sigma}e^{2\pi\img k\l\iota_I^*(\eta),zv\r}.
\end{equation}

Let $\rho:H_I\times H_I\to H_I\times H_I$ be the map defined by
\[ \rho(h_1,h_2)=(\barh_1,\barh_2)=(ah_1-ch_2, -bh_1+dh_2). \]
$\rho$ is bijective and its inverse is given by
\[ \rho^{-1}(\barh_1,\barh_2)=(d\barh_1+c\barh_2,b\barh_1+a\barh_2). \]
Then $av(h_1)-cv(h_2)$ and $-bv(h_1)+dv(h_2)$ are representatives of
$\barh_1$ and $\barh_2$ which we shall denote by $v(\barh_1)$ and
$v(\barh_2)$ respectively.

In view of \eqref{eq:exp1} and \eqref{eq:4-2},
the right hand side of \eqref{eq:brA2} is equal to
\begin{equation}\label{eq:brA3}
 \begin{split}
 &e^{-2\pi\img\left(c\l u,v\r z\sigma\right)}
 e^{-2\pi\img\l\iota_I^*(\eta),ckzv\r}\cdot \\
 & \prod_{i\in I}\quad e^{2\pi\img\l d_iu_i^I,(c\tau+d)\sigma v(\barh_1)\r}
 \phist(-\l d_i\u,zv-v(\barh_1)\tau+v(\barh_2)\r,
 \tau,(c\tau+d)d_i\sigma).
 \end{split}
\end{equation}
Summing up over $(h_1,h_2)$ is the same as summing up over
$(\barh_1,\barh_2)$. Hence from \eqref{eq:brA3} we get
\eqref{eq:hatA} with $h_i$ replaced by $\barh_i$ for $i=1,2$.
This proves Lemma \ref{lemm:hatA}.
\end{proof}

\begin{lemm}\label{lemm:hatnopole}
Assume that $\xi$ satisfies \eqref{eq:xiNdivide} with $\eta$ $T$-Cartier.
Then, for fixed $\tau$, the meromorphic
function $(\hatvarvst)^A(\Delta,\V;z,\tau,\sigma)$ in
$z$ with $\sigma=\frac{k}{N},\ 0<k<N$, has no poles at $z\in \R$.
\end{lemm}

\begin{proof}
The expression \eqref{eq:hatA} in Lemma \ref{lemm:hatA} of the
function $(\hatvarvst)^A(\Delta,\V;z,\tau,\sigma)$ can be rewritten
in the following form as in the case of \eqref{eq:hatvarv2}.
\begin{equation*}
\small
 \begin{split}
 &(\hatvarvst)^A(\Delta,\V,\xi;z,\tau,\sigma) \\
  &=e^{-2\pi\img(c\l u,v\r z\sigma)}
   \sum_{k=0}^n\sum_{K\in\sigmk,h_1\in \hat{H}_K}
   e^{2\pi\img\l u^K(\xi), (c\tau+d)\sigma v(h_1)\r}\cdot \\
   &\sum_{I\in \Sigma_K^{(n-k)}}\frac{w(I)}{|H_I|}
   e^{-2\pi\img\l \iota_I^*(\eta), ckzv\r}\!\!
   \sum_{h_2\in H_I}
    \prod_{i\in I}
    \phist\bigl(-\l u_i^I, zv-\tau v(h_1)+v(h_2)\r, \tau, (c\tau+d)d_i\sigma\bigr).
\end{split}
\end{equation*}
Hence, in order to prove Lemma \ref{lemm:hatnopole}, it is
sufficient to prove that
\[
\small
\begin{split}
\sum_{I\in \Sigma_K^{(n-k)}}\frac{w(I)}{|H_I|}
   e^{-2\pi\img\l \iota_I^*(\eta), ckzv\r}\!\!
   \sum_{h_2\in H_I}
    \prod_{i\in I}
    \phist\bigl(-\l u_i^I, zv-\tau v(h_1)+v(h_2)\r, \tau,
    (c\tau+d)d_i\sigma\bigr),
\end{split}
\]
or, replacing $ck$ by $m$ and $(c\tau+d)\sigma$ by $\sigma$,
\begin{equation}\label{eq:brvarv3}
\sum_{I\in \Sigma_K^{(n-k)}}\frac{w(I)}{|H_I|}\sum_{h_2\in H_I}
   e^{-2\pi\img\l \iota_I^*(\eta), mzv\r}\prod_{i\in I}
    \phist(-\l u_i^I, zv-\tau v(h_1)+v(h_2)\r, \tau,d_i\sigma)
\end{equation}
has no poles at $z\in \R$
for any fixed $K\in \sigmk$ and $h_1\in \hatH_K$.

Note that
$\l \iota_I^*(\eta),v(h_2)\r$ is an integer for all $I\in \sigmn$
since $\eta$ is $T$-Cartier. Therefore
\[ e^{-2\pi\img\l \iota_I^*(\eta), mzv\r}=
 e^{-2\pi\img\l \iota_K^*(\eta),m\tau v(h_1)\r}
 e^{-2\pi\img\l \iota_I^*(\eta),m(zv-\tau v(h_1)+v(h_2))\r}. \]
Note further that
\[ e^{-2\pi\img\l \iota_I^*(\eta),m(zv-\tau_1 v(h_1)+v(h_2))\r}=
 \chi_I(\iota_I^*(m\eta), h_2)e^{-2\pi\img \l \iota_I^*(m\eta),zv\r}
 q_1^{\l\iota_I^*(m\eta),v (h_1)\r}. \]
By a similar argument to the proof of Theorem \ref{theo:laurent},
we see that \eqref{eq:brvarv3} can be expanded
in the form
\[ e^{-2\pi\img\l \iota_K^*(\eta),m\tau v(h_1)\r}
 \sum_{s=0}^\infty (\hatvarst)^A_{K,h_1,s}(z) q^s, \]
where $(\hatvarst)^A_{K,h_1,s}(z)$ belongs to $R(S^1)\otimes\C$.
From this we can conclude that \eqref{eq:brvarv3} has no poles at $z\in \R$.
We refer to Lemma in Section 7 of \cite{Hir} for details.
See also Lemma 4.3 of
\cite{DLM} whose argument can be applied to prove that the function
$(\hatvarvst)^A(\Delta,\V,\xi;z,\tau,\sigma)$
with $\sigma=\frac{k}{N},\ 0<k<N$, is holomorphic on
$\R\times \mathcal{H}$ as a function of $(z,\tau)$. This finishes
the proof of Lemma \ref{lemm:hatnopole}.
\end{proof}

We now proceed to the proof of Theorem \ref{theo:hatrigid}. We follow
\cite{L} for the idea of proof.  We first
show that $\hatvarvst(\Delta,\V,\xi;z,\tau,\sigma)$ is a constant as
a function of $z$.

We regard
$\hatvarvst(\Delta,\V,\xi;z,\tau,\sigma)$ as a meromorphic function of $z$.
By the transformation law \eqref{eq:transformula1} $\phist(z,\tau,\sigma)$
is an elliptic function in $z$ with respect to the lattice
$\Z\cdot N\tau\oplus \Z$ for $\sigma=\frac{k}{N}$ with $0<k<N$.
Hence $\hatvarvst(\Delta,\V,\xi;z,\tau,\sigma)$ with
$\sigma=\frac{k}{N},\ 0<k<N,$ is also
an elliptic function in $z$. Thus, in order to show
that $\hatvarvst(\Delta,\V,\xi;z,\tau,\sigma)$ is a
constant function it suffices to show that it does not have a pole.

Assume that $z$ is a pole. Then $1-t^mq^r\alpha =0$ for some
integer $m\not=0$, some rational number $r$ and a root of unity
$\alpha$. Consequently there are integers $m_1\not=0$ and $k_1$
such that $ m_1z+k_1\tau \in\R$. Then there is an element $A=
\begin{pmatrix} a & b \\ c & d \end{pmatrix} \in SL_2(\Z)$ such
that
\[ \frac{z}{c\tau+d}\in \R. \]
Since
\[
\small
\begin{split}
 \hatvarvst(\Delta,\V,\xi;z,\tau,\sigma)=
 \hatvarvst(\Delta,\V,\xi;A^{-1}({\textstyle\frac{z}{c\tau+d}},A\tau),\sigma)=
 (\hatvarvst)^{A^{-1}}(\Delta,\V,\xi;{\textstyle\frac{z}{c\tau+d}},A\tau,\sigma),
\end{split}
\]
the function
$(\hatvarvst)^{A^{-1}}(\Delta,\V,\xi;w,A\tau,\sigma)$ must have a
pole $w=\frac{z}{c\tau+d}\in \R$. But this contradicts Lemma
\ref{lemm:hatnopole}. This contradiction proves that
$\hatvarvst(\Delta,\V,\xi;z,\tau,\sigma)$ can not have a pole.

Since $\hatvarvst(\Delta,\V,\xi;z,\tau,\sigma)$ is a constant function in $z$
for every generic vector $v\in L$,
$\hatvarst(\Delta,\V,\xi;\tau,\sigma)$ is constant as a function
on $T$. This proves the first part of Theorem \ref{theo:hatrigid}.

To prove the second part note that
\[ \hatvarst(\Delta,\V,\xi;\tau,\sigma)=
 \hatvarvst(\Delta,\V,\xi;z,\tau,\sigma) \]
for any $v$ as constants.
On the other hand, using \eqref{eq:transformula1} and the fact that
$\l u^I(\xi),v\r \equiv h(v) \bmod N\Z$ for any $I\in \sigmn$, we have
\begin{equation}\label{eq:hv}
\begin {split}
 \hatvarst(\Delta,\V,\xi;\tau,\sigma)=
 \hatvarvst(\Delta,\V,\xi;z+\tau,\tau,\sigma)&=
 \zeta^{h(v)}\hatvarvst(\Delta,\V,\xi;z,\tau,\sigma) \\
 &=\zeta^{h(v)}\hatvarst(\Delta,\V,\xi;\tau,\sigma).
\end{split}
\end{equation}
We choose a generic vector $v$ such that $h(v)\not\equiv0$ in
$\Q/N\Z$, which is
possible by Lemma \ref{lemm:hv} below because of the assumption
$\xi\not\in NH_T^2(\Delta,\V)$. Then
$\zeta^{h(v)}$ is not equal to $1$. Hence
\eqref{eq:hv} implies that the constant $\hatvarst(\Delta,\V,\xi;\tau,\sigma)$
must vanish. This finishes the proof of Theorem \ref{theo:hatrigid}.

\begin{lemm}\label{lemm:hv}
Assume that $d_i$ does not belong to $N\Z$ for some $i$ and $\xi$
satisfies the condition \eqref{eq:xiNdivide}.
Then there exists
a generic vector $v\in L$ such that $h(v)$ is non-zero in $\Q/N\Z$.
\end{lemm}
\begin{proof}
Fix an element $I\in \sigmn$ containing $i_0$ with $d_{i_0}\not\in N\Z$.
Since the $\u$ form a basis of
$L_{I,\V}^*$ and $d_{i_0}\not\in \Q/N\Z$, there is a
$v=\sum_{i\in I}a_iv_i\in L_{I,\V}\subset L$ with
$a_i\in \Z,\ a_i\not=0$ for all $i\in I$ such that $\l u^I(\xi),v\r=
\sum_{i\in I}a_id_i$ does
not belong to $N\Z$, i.e., $\l u^I(\xi),v\r\not=0$ in $\Q/N\Z$.
Since $\xi$ satisfies the condition \eqref{eq:xiNdivide}, the value
$\l u^I,v\r \in \Q/N\Z$ is independent of $I$ and equal to $h(v)$ as
remarked above.
\end{proof}

For the proof of Theorem \ref{theo:varrigid} we apply
Theorem \ref{theo:hatrigid} with $\eta=0$ and arbitrary $N>1$ such that
$\xi=u$ does not belong to $NL_\V^*$.
One sees that
$\hatvarst(\Delta,\V,\xi;\tau,\sigma)$ vanishies for
$\sigma=\frac{k}{N},\ 0<k<N$. Since this is true for infinite many
integers $N$ we see that $\hatvarst(\Delta,\V,\xi;\tau,\sigma)$ must
be equal to $0$.


\section{A character formula for orbifold elliptic genus}
In this section we shall give a character formula for
$\hatvar(\Delta,\V,\xi)$. A similar formula was first given by Borisov and
Libgober in \cite{BL1}. The formula for the orbifold elliptic genus of
a multi-fan was given in \cite{HM2}.

Recall that $f_{h,i}$ for $h\in H_J$ and $i\in J$ is given by
$f_{h,i}=\l u_i^J,v(h)\r$ where $v(h)$ satisfies
\[ 0\leq\l u_i^J,v(h)\r<1 \ \text{for all $i\in J$}. \]

\begin{theo}\label{thm:BL2}
Let $(\Delta,\V,\xi)$ be a triple of complete simplicial multi-fan in a lattice $L$ of rank $n$, a set of generating edge vectors and a $\Q$-divisor. Then
\begin{equation}\label{algn:BL}
 \begin{split}
 &\hatvarst(\Delta,\V,\xi) \\
 =&\sum_{u\in L^*}t^{-u}\left(\sum_{k=0}^n
  \sum_{J\in\Sigma^{(k)}, h\in H_J}(-1)^{n-k}
  \deg(\Delta_J)\zeta^{f_{J,h}(\xi)}q^{\l u,v_{J,h}\r}
  \prod_{i\in J}\frac1{1-\zeta^{d_i} q^{\l u,v_i \r}}\right),
 \end{split}
\end{equation}
where $f_{J,h}(\xi)=\sum_{i\in J}d_if_{h,i}$ and
$v_{J,h}=\sum_{i\in J}f_{h,i}v_i$.
\end{theo}
\begin{note}
$v_{J,h}=\sum_{i\in J}\l u_i^J,v(h)\r v_i=v(h)\in L_J$ for a particularly
chosen representative $v(h)$. Hence $\l u,v_{J,h}\r\in \Z$.
This shows that $\hatvarst(\Delta,\V,\xi)$ belongs to
$(R(T)\otimes \C)[[q]]$. This fact
was already proved in Theorem \ref{theo:laurent}.
Note that
\[\zeta^{f_{J,h}(\xi)}q^{\l u,v_{J,h}\r}
  \prod_{i\in J}\frac1{1-\zeta^{d_i} q^{\l u,v_i \r}}=\prod_{i\in J}
  \frac{(\zeta^{d_i} q^{\l u,v_i\r})^{f_{h,i}}}{1-\zeta^{d_i} q^{\l u,v_i \r}}. \]
\end{note}

We need the following three lemmas.
\begin{lemm}\label{lemm:BL1}
Suppose that $|q|<|t|<1$. Then we have the equality
\begin{equation}\label{eq:BL1}
 \phist(z,\tau,\sigma)=\sum_{m\in \Z}\frac{t^m}{1-\zeta q^m}.
\end{equation}
\end{lemm}

\begin{lemm}\label{lemm:BL}
Put $\alpha=e^{2\pi\img w}$. Suppose that $|\alpha|=1$
and $|q|<|t|<1$.
If $l\not=0$ is an integer, then we have the equality
\[ \phist(lz+w,\tau,\sigma)=
\begin{cases}
 \sum_{m\in \Z}\alpha^mt^{lm}\dfrac1{1-\zeta q^m} &
 \text{if \ $l>0$}, \\
 \sum_{m\in \Z}\alpha^mt^{lm}
 \left(\dfrac1{1-\zeta q^m}-1\right) & \text{if \ $l<0$}.
\end{cases} \]
\end{lemm}

\begin{lemm}\label{lemm:1}
Let $f$ be a real number with $0<f<1$ and $l\not=0$ an integer.
If
\[ |q^f|,|q^{1-f}|<|t^{|l|}|,\ |t|\leq 1 \]
then
\begin{equation*}\label{eqn:1}
 \phist(lz+f\tau+w,\tau,\sigma)=
 \sum_{m\in \Z}\alpha^m t^{lm}
 \frac{q^{fm}}{1-\zeta q^m},
\end{equation*}
where $\alpha=e^{2\pi\img w},\ |\alpha|=1$ as before.
\end{lemm}

Lemma \ref{lemm:BL1} is essentially the same as Lemma 3.5 of
\cite{HM2}. The equality \eqref{eq:BL1} was first proved in \cite{BL1}.
Lemma \ref{lemm:BL} is essentially the same as Lemma 3.6 of \cite{HM2}.
Lemma \ref{lemm:1} is essentially the same as Lemma 3.7 of \cite{HM2}.

We now proceed to the proof of Theorem \ref{thm:BL2}. Take a generic
vector $v\in L_\V$. Then $\l u_i^I, v\r$ is an integer for
any $I\in \sigmn$.
For $I\in \sigmn$ we put $I(v)=\{i\mid \l u_i^I,v\r<0\}$.
We take a representative $v(h_1)$ of $h_1\in \hat{H}_K$
such that
\[ \l u_i^I,v(h_1)\r=\l u_i^K,v(h_1)\r =f_{h_1,i}\]
for each $i\in K$. Then, by \eqref{eq:hatvarv2}, $\hatvarvst(\Delta,\V,\xi)$
can be written in the form
\begin{equation}\label{eq:hatvar3}
\begin{split}
 \hatvarvst(\Delta,\V,\xi)=&\sum_{k=0}^n\sum_{K\in\Sigma^{(k)},h_i\in \hatH_K}
 \sum_{I\in \Sigma_K^{(n-k)}}\frac{w(I)}{|H_I|} \\
 &\prod_{i\in I\setminus K}
\phist(-\l u_i^I,zv+v(h_2)\r,\tau,d_i\sigma) \\
&\prod_{i\in K}\zeta^{d_if_{h_1,i}}
\phist(-\l u_i^I,zv-\tau v(h_1)+v(h_2)\r,\tau,d_i\sigma).
\end{split}
\end{equation}

If $t\in \C$ satisfies
\[ |q^{f_{h_1,i}}|,|q^{1-f_{h_1,i}}|<|t^{|\uv|}|, \quad |t|<1, \]
then by Lemma \ref{lemm:1}, we have
\[ \begin{split}
 \phist(-&\l u_i^I,zv-\tau v(h_1)+v(h_2)\r,
\tau, d_i\sigma) \\
&=\sum_{m_i\in\Z}\chi_I(u_i^I,v(h_2))^{-m_i}q^{m_if_{h_1,i}}t^{-m_i\uv}
\frac1{1-\zeta^{d_i} q^{m_i}}.
\end{split} \]

Next fix $i\in I\setminus K$. Then, by Lemma \ref{lemm:BL}, we have
\[ \begin{split}
 &\phist(-\l u_i^I,zv+v(h_2)\r,\tau,d_i\sigma) \\
 &= \begin{cases}
  \sum_{m_i\in\Z}\chi_I(u_i^I,v)^{-m_i}t^{-m_i\uv}\dfrac1{1-\zeta^{d_i} q^{m_i}}
  \qquad \text{for}
  \ i\in I(v)\setminus K, \\
  \sum_{m_i\in\Z}\chi_I(u_i^I,v)^{-m_i}t^{-m_i\uv}
 \left(\dfrac1{1-\zeta^{d_i} q^{m_i}}-1\right)
  \quad \text{for} \ i\in I\setminus (I(v)\cup K).
\end{cases}
\end{split}
\]

Now suppose that $|t|<1$ and $q$ satisfy
\[ |q^{f_{h_1,i}}|, |q^{1-f_{h_1,i}}|<|t^{|\l u_i^I,v\r|}| \]
for all $i\in K,\ h_1\in \hat{H}_K, \ K\in \sigmk$,
and
\[ |q|<|t^{|\l u_i^I,v\r|}| \]
for all $i\in I\setminus K, I\in \Sigma_K^{(n-k)}$ and
$K\in \sigmk$. Then we obtain

\begin{equation*}\label{eq:new1}
\begin{split}
\prod_{i\in I\setminus K}\!\! \phist(-&\l
u_i^I,zv+v(h_2)\r,\tau,d_i\sigma)\!\! \prod_{i\in
K}\zeta^{d_if_{h_1,i}}
\phist(-\l u_i^I,zv-\tau v(h_1)+v(h_2)\r,\tau,d_i\sigma) \\
 &=\zeta^{f_{K,h_1}(\xi)}\prod_{i\in I(v)\setminus K}{\sum_{m_i\in \Z}
 \chi_I(u_i^I,h_2)^{-m_i}t^{-m_i\uv}\frac{1}{1-\zeta^{d_i} q^{m_i}}} \\
 &\quad \prod_{i\in I\setminus (I(v)\cup K)}{\sum_{m_i\in \Z}
 \chi_I(u_i^I,h_2)^{-m_i}
 t^{-m_i\uv}\left(\frac{1}{1-\zeta^{d_i} q^{m_i}}-1\right)} \cdot \\
&\quad \quad \ \prod_{i\in K}{\sum_{m_i\in \Z}
\chi_I(u_i^I,h_2)^{-m_i}q^{m_if_{h_1,i}}
 t^{-m_i\uv}\frac{1}{1-\zeta^{d_i} q^{m_i}}}.
\end{split}
\end{equation*}

Furthermore
\begin{equation*}\label{eq:sumJ}
\begin{split}
\prod_{i\in I(v)\setminus K}&\frac1{1-\zeta^{d_i}q^{m_i}}
\prod_{i\in I\setminus (I(v)\cup K)}
\left(\frac1{1-\zeta^{d_i} q^{m_i}}-1\right)
\prod_{i\in K}\frac1{1-\zeta^{d_i} q^{m_i}} \\
=&\sum_{l=0}^n\sum_{J\in \Sigma^{(l)} :(I(v)\cup K)\subset J\subset I}
(-1)^{n-l}\prod_{j\in J}\frac1{1-\zeta^{d_i} q^{m_j}},
\end{split}
\end{equation*}
for each $I\in \Sigma_K^{(n-l)}$.

If we put $u=\sum_{i\in I}m_iu_i^I \in L_{I,\V}^*$, then
$m_i=\langle u,v_i\rangle$.
Hence $\prod_{i\in I}t^{-m_i\langle u_i^I,v\rangle} =t^{-\langle
u,v\rangle}$. Since $\chi_I(u,\ )=e^{2\pi\sqrt{-1}\langle u,\
\rangle}$ we see that $\prod_{i\in
I}\chi_I(u_i^I,h)^{-m_i}=\chi_I(u,h)^{-1}$. The $1$-dimensional
representation $\chi_I(u,h)^{-1}$ of $H_I=L/L_{I,\V}$ is trivial
if and only if $u\in L^*$. It follows that
\[ \sum_{h\in H_I}\chi_I(u,h)^{-1}=
\begin{cases}
  |H_I|\quad &\text{if}\ u\in L^*, \\
  0 &\text{if}\ u\not\in L^*.
\end{cases} \]

Combining these we have
\begin{equation}\label{eq:new2}
\small
 \begin{split}
 \sum_{h_2\in H_I}
 &\prod_{i\in I\setminus K}\!\!
\phist(-\l u_i^I,zv+v(h_2)\r,\tau,d_i\sigma)
 \!\!\prod_{i\in K}\zeta^{d_if_{h_1,i}}
\phist(-\l u_i^I,zv-\tau v(h_1)+v(h_2)\r,\tau,d_i\sigma)  \\
 =&|H_I|\!\!\sum_{u\in L^*}t^{-\langle u,v\rangle}
  \zeta^{f_{K,h_1}(\xi)}q^{\l u,v_{K,h_1}\r}
  \Bigl(
   \sum_{k=0}^n\sum_{J\in \Sigma^{(l)} :(I(v)\cup K)\subset J\subset I}
  \!\!\!\!(-1)^{n-l}
   \prod_{j\in J}\frac1{1-\zeta^{d_i} q^{\langle u,v_j\rangle}}\Bigr).
\end{split}
\end{equation}

Fix $J\in \Sigma$. It is easy to see that the union of
$\{\hat{H}_K\mid K\in \Sigma, K\subset J\}$ is disjoint.
Since any $h\in H_J$ is contained in
$\hat{H}_{K_h}$ by (\ref{eq:0}) where
$K_h=\{j\in J\mid f_{h,j}\not=0\}$, we have
$H_J=\sqcup \hat{H}_K$.
Moreover we have
\[ f_{J,h}(\xi)=f_{K_h,h}(\xi) \quad\text{and}\quad v_{J,h}=v_{K_h,h}.\]

Taking these facts in account in \eqref{eq:new2} and
using \eqref{eq:hatvar3},  we get
\begin{equation}\label{eq:BLv}
\small
\begin{split}
&\hatvarvst(\Delta,\V,\xi) \\
=& \sum_{u\in L^*}t^{-\l u,v\r}\sum_{k=0}^n\sum_{J\in \sigmk}
\sum_{h\in H_J}\zeta^{f_{J,h}(\xi)}q^{\l u,v_{J,h}\r}
(-1)^{n-k}\sum_{I:I(v)\subset J\subset I}w(I)
\prod_{j\in J}\frac1{1-\zeta^{d_i} q^{\langle u,v_j\rangle}}.
\end{split}
\end{equation}

Since $\sum_{I\in \sigmn :I(v)\subset J\subset I}w(I)=
\deg(\Delta_J)$ by definition we have
\begin{equation}\label{eq:BLv2}
 \begin{split}
 &\hatvarvst(\Delta,\V,\xi) \\
 =
 &\sum_{u\in L^*}t^{-\l u,v\r}\biggl(\sum_{k=0}^n
 \sum_{J\in\Sigma^{((k)},h\in H_J}(-1)^{n-k}
  \deg(\Delta_J)\zeta^{f_{J,h}(\xi)}q^{\l u,v_{J,h}\r}
  \prod_{i\in J}\frac1{1-\zeta^{d_i} q^{\l u,v_i \r}}\biggr).
\end{split}
\end{equation}

Since $\hatvarst(\Delta,\V,\xi)$ belongs to
$(R(T)\otimes\C])[[q]]$
and \eqref{eq:BLv2} holds for any generic vector $v$,
Theorem \ref{thm:BL2} follows.

\begin{rema}
Stabilized orbifold elliptic genus along a vector $v$ is defined for
a complete simplicial multi-fan by the formula \eqref{eq:hatvarv}.
However the right-hand side of \eqref{eq:hatvarv} still has meaning for
general simplicial multi-fans. Thus we can define $\hatvarvst(\Delta,\V,\xi)$
in general by \eqref{eq:hatvarv}. It can be written in the form
\eqref{eq:hatvarv2} too.
Note that Theorem \ref{theo:laurent}
does not hold in general because it depends on an integrality theorem
(Lemma 2.5 in \cite{HM2}) that holds only for complete simplicial multi-fans.
Neither holds Theorem \ref{thm:BL2} in general. But the following Proposition
holds.
\end{rema}

\begin{prop}\label{prop:BL2}
 Let $\Delta$ be a simplicial fan satisfying
the condition that every $J\in \Sigma$ is contained in some $I\in \sigmn$.
Then \eqref{eq:BLv2} holds provided that the generic vector $v$
is contained in $\bigcup_{I\in \sigmn}C(I)$ :
\begin{equation*}\label{eq:BLv2b}
 \begin{split}
 &\hatvarvst(\Delta,\V,\xi) \\
 =
 &\sum_{u\in L^*}t^{-\l u,v\r}\left(\sum_{k=0}^n
 \sum_{J\in\Sigma^{(k)},h\in H_J}(-1)^{n-k}
  \zeta^{f_{J,h}(\xi)}q^{\l u,v_{J,h}\r}
  \prod_{i\in J}\frac1{1-\zeta^{d_i} q^{\l u,v_i \r}}\right).
\end{split}
\end{equation*}
\end{prop}

\begin{proof}
Note first that $w(I)=1$ for all $I\in \sigmn$ since
$\Delta$ is a fan. It suffices to show that
\[ \#(S_v(J))= \#\{I\mid I(v)\subset J\subset I\}=1 \]
for all $J\in \Sigma$ provided that $v$
is contained in $\bigcup_{I\in \sigmn}C(I)$ because the completeness of
$\Delta$ was only used at the step from \eqref{eq:BLv} to
\eqref{eq:BLv2}. Suppose that $J$ lies in
$\Sigma^{(l)}$. Then the set $\{I\mid I(v)\subset J\subset I\}$ is identified
with the set $S_v(J)=\{I\in \Sigma_J^{(n-l)}\mid \bar{v}\in C_J(I)\}$
where $\bar{v}$ is the image of $v$ in $L^J_\R$. But
$\Delta_J$ is also a fan and every simplex $K\in \Sigma_J$ is contained in
some $I\in \Sigma^{(n-l)}$. Moreover $\bar{v}$ is generic and contained in
$\bigcup_{I\in \Sigma_J^{(n-l)}}C_J(I)$. Hence $\#(S_v(J))=1$.
\end{proof}

\begin{rema}\label{rem:BL}
If we consider $\hatvarst(\Delta,\V,\xi)$ as a function
assigning to each generic vector $v\in \bigcup_{I\in \sigmn}C(I)$
the value $\hatvarvst(\Delta,\V,\xi)$, then \eqref{algn:BL} holds
in this case with the understanding that $t^{-u}$ is a function
assigning the value $t^{\l u,v\r}$ to $v$.
\end{rema}


\section{Invariance of orbifold elliptic genus under
birational morphisms of multi-fans}

Let $\Delta$ be a simplicial multi-fan (not necessarily complete)
in a lattice $L$ and $\V=\{v_i\}_{i\in \sigmone}$ a set of prescibed
edge vectors as before.
We shall consider a multi-fan $\delpr=(\sigmpr,C^\prime,
w^{\prime\pm})$ in the same lattice $L$ and a
morphism $\rho:\delpr \to \Delta$ related to $\Delta$ in the following way.
$\rho$ may be regarded as a generalization of birational morphism in
toric theory. We require the following conditions for $\delpr$:
\begin{enumerate}
\renewcommand{\labelenumi}{\alph{enumi})}
\item There is an injection $\kappa:\sigmone\to\Sigma^{\prime(1)}$
satisfying
$C(i)=C^\prime(\kappa(i))$ for each $i \in \sigmone$.
\item For each $J^\prime \in \sigmpr$ there is a simplex $J\in \Sigma$
such that $\Cpr(\Jpr)\subset C(J)$. Moreover, for each $J\in \Sigma$,
the collection $\{\Cpr(\Jpr)\mid \Jpr\in \sigmpr, \Cpr(\Jpr)\subset C(J)\}$
gives a subdivision of the cone $C(J)$.
We shall denote by $\rho(\Jpr)$ the minimal simplex $J\in \Sigma$ such that
$\Cpr(\Jpr)\subset C(J)$.
\item For $\Ipr\in \sigmprn$
\[ w^{\prime\pm}(\Ipr)=w^\pm(\rho(\Ipr)). \]
In particular $w^\prime(\Ipr)=w(\rho(\Ipr))$.
\end{enumerate}
$\rho(\{\ipr\})$ is
simply denoted by $\rho(\ipr)$. Note that $\rho(\kappa(i))=i$ for
$i\in \sigmone$. The map $\rho:\sigmpr\to \Sigma$ is sometimes denoted
by $\rho:\delpr\to \Delta$.
Let $\V=\{v_i\}_{i\in \sigmone}$ and $\Vpr=\{v_{\ipr}\}_{\ipr \in \sigmprone}$
be sets of edge vectors for $\Delta$ and
$\delpr$ respectively. Note that $v_{\kappa(i)}$ and $v_i$ lie on the
same half line $C(i)$ but they may be different.

The vector $v_{\ipr}$ is written uniquely in the form
\begin{equation*}\label{eq:rho1}
 v_{\ipr}=\sum_{i\in \rho(\ipr)}a_{\ipr i}v_i
\end{equation*}
with $a_{\ipr i}\in \Q_{>0}$. We put $a_{\ipr i}=0$ for
$i\not\in \rho(\ipr)$. We then define a map
$\rho^* :H_T^2(\Delta,\V)\otimes\Q\to H_T^2(\delpr,\Vpr)\otimes\Q$ by
\begin{equation*}\label{eq:rho2}
 \rho^*(x_i)=\sum_{\ipr\in \sigmprone}a_{\ipr i}x_{\ipr}.
\end{equation*}

\begin{rema}\label{rema:xiquot}
As a special case related to \eqref{eq:ramif} consider the following
situation: $\delpr=\Delta$ and $\Vpr=\{v_i^\prime\}_{i\in \sigmone},
\V=\{v_i\}_{i\in \sigmone}$ with $v_i^\prime=a_iv_i,\ a_i>0$. If
$\xi=\sum_ic_ix_i$, then $\rho^*(\xi)=\sum_ia_ic_ix_i$.
In particular $\rho^*(\xi)=\sum_ix_i$ if and only if
$\xi=\sum_i\frac1{a_i}x_i$. In this case $D_\Delta$ defined in
Remark \ref{rema:singEll} is given by $\sum_i\frac{a_i-1}{a_i}x_i$.
As typical examples of such a situation one can quote weighted
projective spaces $\Proj(a_0,\ldots,a_n)$. We assume that
the greatest common divisor of $\{a_0,\ldots, \widehat{a_i},\ldots,a_n\}$
is equal to $1$ for all $0\leq i\leq n$. $\Proj(a_0,\ldots,a_n)$ has
two natural orbifold structures; one is given as the quotient space
of $\C^{n+1}\setminus \{0\}$ by the action of $\C^*$ given by
\[ z(z_0,\ldots,z_n)=(z^{a_0}z_0,\ldots,z^{a_n}z_n), \]
and the other as a global quotient of $\Proj^n$ by the standard action
of the group $\Z/a_0\Z\times \cdots\times \Z/a_n\Z$. The former corresponds
to $\V$ and the latter to $\Vpr$.
\end{rema}

\begin{lemm}\label{lemm:rhostar}
$\rho^*$ extends to a ring homomorphism
$\rho^* :H_T^*(\Delta,\V)\otimes\Q\to H_T^*(\delpr,\Vpr)\otimes\Q$.
Moreover $\rho^*(u)=u$ for $u\in L_\Q^*$ and
$\rho^*$ is an $H_T^*(BT)\otimes\Q$-module map. It satisfies
\begin{equation}\label{eq:iotarho}
 \iota_{\Ipr}^*\circ \rho^*=\iota_I^*
\end{equation}
for any $I\in \sigmn$ and $\Ipr\in \sigmprn$ such that $\rho(\Ipr)=I$.
\end{lemm}
\begin{proof}
Let $I$ be a non-empty subset of $\sigmone$. In order to show that
$\rho$ extends to a ring homomorphism it is enough
to show that $\prod_{i\in I}x_i=0$ implies
\[ \prod _{i\in I}\rho^*(x_i)=
 \prod _{i\in I}\sum_{\ipr\in \sigmprone}a_{\ipr i}x_{\ipr}=0. \]
Assume that there is an element $\ipr(i)\in \sigmprone$ for each $i\in I$
such that $i\in\rho(\ipr(i))$ and $\prod_{i\in I}x_{\ipr(i)}\not=0$.
Then $\Ipr=\{\ipr(i)\mid i\in I\}$ is a simplex and $\rho(\Ipr)=I$.
Hence $I$ must be a simplex. This contradicts the fact
that $\prod_{i\in I}x_i=0$. Therefore, for all subsets
$\{\ipr\mid i\in\rho(\ipr),i\in I\}$, the product $\prod x_{\ipr}$
must be equal to $0$. Then
$\prod _{i\in I}\sum_{\ipr\in \sigmprone}a_{\ipr i}x_{\ipr}$, as
a linear combination of such elements, is equal to $0$.

For $u\in L_\Q^*$ we have
\begin{align*}
 \rho^*(u)&=\sum_i\l u,v_i\r\rho^*(x_i) \\
 &=\sum_i\l u,v_i\r\sum_{\ipr}a_{\ipr i}x_{\ipr} \\
 &=\sum_{\ipr}\l u,v_{\ipr}\r x_{\ipr}=u.
\end{align*}
Then $\rho^*(ux)=\rho^*(u)\rho^*(x)=u\rho^*(x)$ for any
$x\in H_T^*(\Delta,\V)$.
This shows that $\rho^*$ is an $H_T^*(BT)\otimes\Q$-module homomorphism.

In order to prove \eqref{eq:iotarho} we may check it on the generators
$x_i\in H_T^2(\Delta)$ since $\iota_I^*$, $\iota_{\Ipr}^*$
and $\rho^*$ are ring homomorphisms. If $v_{\ipr}=\sum_{i\in I}a_{\ipr i}v_i$
for $\ipr\in \Ipr$, then $u_{\ipr}^{\Ipr}=\sum_{i\in I}b_{\ipr i}u_i^I$ with
$\sum_{\ipr\in \Ipr}a_{\ipr i}b_{\ipr j}=\delta_{ij}$. Hence
\begin{align*} \iota_{\Ipr}^*(\rho^*(x_i))
&=\iota_{\Ipr}^*(\sum_{\ipr}a_{\ipr i}x_{\ipr}) \\
&=\sum_{\ipr}a_{\ipr i}u_{\ipr}^{\Ipr} \\
&=\sum_{\ipr}\sum_ja_{\ipr i}b_{\ipr j}u_j^I=u_i^I=\iota_I^*(x_i).
\end{align*}
This proves \eqref{eq:iotarho}.
\end{proof}

If $\xi=\sum_{i\in \sigmone}d_ix_i$ and $\rho^*(\xi)=
\sum_{\ipr\in \sigmprone}d_{\ipr}x_{\ipr}$, then
\begin{equation*}\label{eq:dipr}
 d_{\ipr}=\sum_{i\in \rho(\ipr)}a_{\ipr i}d_i
\end{equation*}
as is easily seen. If we put $u^J(\xi)=\sum_{i\in J}d_iu_i^J$ with
$J=\rho(\ipr)$, then $d_{\ipr}$ is also written as
\begin{equation}\label{eq:dipr2}
 d_{\ipr}=\l u^J(\xi),v_{\ipr}\r.
\end{equation}
Note also
\[ d_i=\l u^J(\xi),v_i\r.  \]
\begin{theo}\label{thm:crepant1}
Let $\Delta$ be a complete simplicial multi-fan and let
$\rho:\delpr \to \Delta$ be a map satisfying a), b) and c). We
put $\xipr=\rho^*(\xi)$. Then $\delpr$ is also complete and
the following equality holds:
\[ \hatvarst(\delpr,\Vpr,\xipr)=\hatvarst(\Delta,\V,\xi). \]
\end{theo}
The proof will be given in the next section by using a local
version of the theorem.

\section{Local version of invariance}
For a triple $(\Delta,\V,\xi)$ with $\xi=\sum_{i\in \sigmone}d_i\xi_i$
we
set
\[ b_J(\Delta,\V,\xi)=\sum_{h\in H_J}(-1)^{n-k}
  \zeta^{f_{J,h}(\xi)}q^{v_{J,h}}
  \prod_{i\in J}\frac1{1-\zeta^{d_i} q^{v_i}} \]
where $q^v$ for $v\in L$ is considered as a function
assigning to each $-u\in L^*$ the value $q^{\l u,v\r}$. Thus
$q^{v_{J,h}}\prod_{i\in J}\frac1{1-\zeta^{d_i} q^{v_i}}$ takes
the value
\[ q^{\l u,v_{J,h}\r}\prod_{i\in J}\frac1{1-\zeta^{d_i} q^{\l u,v_i \r}}.\]
at $-u\in  L^*$.
When $\Delta$ is a complete simplicial multi-fan
Theorem \ref{thm:BL2} says that
\begin{equation}\label{eq:BL2}
 \hatvarst(\Delta,\V,\xi)=\sum_{u\in L^*}t^{-u}\left(
 \sum_{J\in \Sigma}\deg(\Delta_J)b_J(\Delta,\V,\xi)\right)(u).
\end{equation}
In the sequel we shall write
\[ \hatvarst(\Delta,\V,\xi)=
\sum_{J\in \Sigma}\deg(\Delta_J)b_J(\Delta,\V,\xi) \]
to mean the equality \eqref{eq:BL2}.

Let $I\in \sigmn$. For a while we shall concentrate on the part of $\Sigma$
consisting of all faces of $I$, which we shall denote by $\Sigma(I)$.
Similarly
the part of $\Delta$ restricted on $\Sigma(I)$ will be denoted by
$\Delta(I)$. In this case we consider it as a fan forgetting the function
$w^\pm$.
When there is no fear of
confusion, we shall simply denote $\Sigma(I)$ by $\Sigma$ and
$(\Delta(I),\V|I,\xi|I)$ by $(\Delta,\V,\xi)$.

If $\delpr$ is a fan and $\rho:\delpr \to \Delta(I)$ is a map
satisfying a) and b), we put $\xipr=\rho^*(\xi)$.
If $\xipr=\sum_{\ipr\in \sigmprone}d_{\ipr}x_{\ipr}$, then
\[ d_{\ipr}=\sum_{i\in I}a_{\ipr i}d_i. \]
Define
\[ b_J(\delpr,\Vpr,\xipr;\Delta,\V,\xi):=
 \sum_{\Jpr, \rho(\Jpr)=J}b_{\Jpr}(\delpr,\Vpr,\xipr)
 -b_J(\Delta,\V,\xi). \]
for $J\in \Sigma(I)$, i.e., for $J\subset I$.

The following theorem can be considered as a local version of
Theorem \ref{thm:crepant1}.

\begin{theo}\label{thm:Bn}
The statement
\[ \bm{B_n}: \quad b_J(\delpr,\Vpr,\xipr;\Delta,\V,\xi)=0 \quad
\text{for $\dim L=n$ and for all $J\in \Sigma$} \]
holds for $n\geq 1$.
\end{theo}

The case $n=1$ is easy. In this case $\V$ and $\Vpr$ consist of one
vector $v_1$ and $v_{1'}$ respectively. They are of the form
$v_1=a_1v$ and $v_{1'}=a_{1'}v$ with $a_1, a_{1'}\in \Z_{>0}$ and $v$
the primitive integral vector in $C(1)=\Cpr(1')$.
Moreover $H_{1}\cong \Z/a_1$ and $f_{h,1}=\bar{h}/a_1$ for $h\in \Z/a_1$
where $\bar{h}\in \Z$ is the representative of $h$ with $0\leq\bar{h}<a_1$.
Similarly $H_{1'}\cong \Z/a_{1'}$ and $f_{h',1'}=\bar{h}'/a_{1'}$ for
$h'\in \Z/a_{1'}$.  Since $v_{1'}=\frac{a_{1'}}{a_1}v_1$, we have
$\frac{d_{1'}}{a_{1'}}=\frac{d_1}{a_1}$ by definition. We denote this
last number by $d$. Then
\[ b_{\{1\}}(\Delta,\V,\xi)=
\frac{\sum_{\bar{h}=0}^{a_1-1}\zeta^{\frac{\bar{h}}{a_1}d_1}
q^{\frac{\bar{h}}{a_1}v_1}}
{1-\zeta^{d_1}q^{v_1}}
=\frac{1}{1-\zeta^{\frac{1}{a_1}d_1}q^{\frac{1}{a_1}v_1}}
=\frac{1}{1-\zeta^dq^v}.
\]
Similarly we have
\[ b_{\{1'\}}(\delpr,\Vpr,\xipr)=\frac{1}{1-\zeta^dq^v}.
\]
Hence
\[ b_{\{1\}}(\delpr,\Vpr,\xipr:\Delta,\V,\xi)=
 b_{\{1'\}}(\delpr,\Vpr,\xipr)- b_{\{1\}}(\Delta,\V,\xi)=0. \]
Clearly
\[b_{\emptyset}(\delpr,\Vpr,\xipr:\Delta,\V,\xi)=1-1=0. \]
This proves $B_1$.

To illustrate the proof for general case we first give a proof of $B_n$ for
the following special case, namely the case where $v_{\kappa(i)}=v_i$ for
all $i\in \sigmone$ and $\rho(\ipr)=I$ for all
$\ipr \in \sigmprone\setminus \kappa(\sigmone)$.
In this case we introduce a multi-fan $\delstartil$ defined in the
following way. Note that $\kappa$ induces an injective simplicial map
$\kappa:\partial\Sigma=\Sigma\setminus \sigmn\to \sigmpr$ compatible
with $C$ and $\Cpr$.
This simplicial map is realized as the identity map on the boundary of
the cone $C(I)$. A simplicial set $\sigmstartil$ is defined as the sum
of $\sigmpr$ and $\Sigma$ glued along $\partial\Sigma$
via $\kappa$. Note that $\sigmstartil^{(1)}=\sigmprone$.
$C$ and $\Cpr$ determines a cone structure $\Cstartil:\sigmstartil\to L_\R$.
We define the function $w^\pm$ for $\sigmstartil$ as follows:
\begin{align*}
 w^+(I_k)=\begin{cases} 1, & I_k\in \sigmprn, \\
                      0, & I_k\in \sigmn, \text{i.e., $I_k=I$}.
\end{cases}
 \quad
 w^-(I_k)=\begin{cases} 0, & I_k\in \sigmprn, \\
                      1, & I_k=I. \end{cases}
\end{align*}
Then
\begin{align*}
   w(I_k)=\begin{cases}+1, &  \text{for $I_k\in\sigmprn$,} \\
                     -1, &  \text{for $I_k=I$.} \end{cases}
\end{align*}
The triple $(\sigmstartil,\Cstartil,w^\pm)$ defines a simplicial multi-fan
$\delstartil$.

We claim that $\delstartil$ is complete and
\begin{equation}\label{eq:startil}
\deg((\delstartil)_J)=\begin{cases}
      +1, & J\in \sigmpr\setminus \partial\Sigma, \\
      -1, & J\in \sigmn,\ \text{i.e., $J=I$}, \\
      \ \ 0, & J\in \partial\Sigma.
 \end{cases}
\end{equation}

In fact take a generic vector $v$ in $L_\R$ and consider
$d_v(J)=\sum_{I_k\in S_v(J)}w(I_k)$ for $J\in \sigmstartil^{(n-1)}$
where $S_v(J)$ is defined by \eqref{eq:Sv}. If $J$ is contained in
$\sigmpr\setminus \partial\Sigma$ then
$\Cpr(J)$ is not contained in the boundary of $C(I)$ and is two-sided
in $L_\R$. From this it follows easily that $d_v(J)$ is equal to $1$
independently of generic $v$ . This fact shows that $(\delstartil)_J$
is complete and $\deg((\delstartil)_J)=+1$.
On the other hand, if $J$ is contained in $\partial\Sigma$, then
$J$ is a face of $I$ and of exactly one simplex $I_k$ in $\sigmprn$.
Noting that
$w(I)=-1$ and $w(I_k)=+1$ we see that $d_v(J)=0$ independently of
generic $v$. Hence $(\delstartil)_J$
is complete and $\deg((\delstartil)_J)=0$ in this case. Since
$(\delstartil)_J$ is complete for all $J\in \sigmstartil^{(n-1)}$,
the multi-fan $\delstartil$ is complete. Other statements concerning
$(\delstartil)_J$ for $J\in \sigmstartil$ of arbitrary dimensions
can be proved in a similar way.

We define $\Vstartil$ and $\xistartil$ to be equal to $\Vpr$
and $\xipr$ respectively. In view of the assumption
$v_{\kappa(i)}=v_i$ these definitions make sense.
Then we have
\[\xistartil=
 \!\!\!\sum_{i_*\in \sigmstartil^{(1)}}d_{i_*}(\tilde{x}_*)_{i_*}
 =\!\!\!\sum_{\ipr\in \sigmprone}d_{\ipr}x'_{\ipr}
 =\rho^*\Bigl(\sum_{i\in\sigmone}d_ix_i\Bigr)=\rho^*(u)=u
 \in H_T^2(\delstartil,\Vstartil)\otimes\Q
 \]
for $u=u^I(\xi)=\sum_{i\in \sigmone}d_i\u\in L_\V^*$.
We can apply Theorem \ref{theo:varrigid} to obtain
\[ \hatvarst(\delstartil,\Vstartil,\xistartil)=0. \]
By
\eqref{eq:BL2} and \eqref{eq:startil} we have
$\hatvarst(\delstartil,\Vstartil,\xistartil)
=\sum_{u\in L^*}t^{-u}b_I(\delpr,\Vpr,\xipr;\Delta,\V,\xi)(u)$ and finally
\[ b_I(\delpr,\Vpr,\xipr;\Delta,\V,\xi)=0. \]
In this case
\[ b_J(\delpr,\Vpr,\xipr;\Delta,\V,\xi)=0 \]
trivially holds for each $J\in\sigmk$ with $k<n$.
Thus $B_n$ is true in this case.

The proof of general case proceeds by induction starting from $B_1$:
\[ B_1\to C_2\to B_2\to \cdots \to B_{n-1}\to C_n\to B_n\to \cdots .\]
The statement $C_n$ is described by using multi-fans $\delprstar$
and $\delstar$ which will be constructed below from $\delpr$ and
$\Delta$ respectively.

Let $W$ be a rational affine hyperplane in $L_\R$ intersecting
$C(I)$ transversally and let $S=W\cap C(I)$. We assume that there
is an integral point $v\in L$ in the interior of $S$. Let $S_m$ be
the image of the homothety $\psi$ of ratio $m$ centered at $v$
where $m$ is a positive integer.  $S$ and $S_m$ are geometric
simplices. We give a triangulation of $\overline{S_m\setminus S}$
as follows. First the boundary $\partial S$ is triangulated in the
standard way, i.e., its simplices are of the form $S_J=S\cap C(J)$
for $J\in \Sigma(I)^{(k)},0<k<n$. The boundary $\partial S_m$ is
triangulated by the barycentric subdivision, i.e., its simplices
are of the form $(b_{J_1}\ldots b_{J_k})$ with $J_1\subsetneqq
\ldots \subsetneqq J_k\subsetneqq I$. Here $b_{J}$ is the
barycenter of the simplex $\psi(S\cap C(J))$. Simplices between
$\partial S$ and $\partial S_m$ are of the form
$S_J*(b_{J_1}\ldots b_{J_k})$ with $J\subset J_1\subsetneqq
\ldots\subsetneqq J_k\subsetneqq I$ where $*$ denotes the join. A
triangulation of $S_m$ is induced from this triangulation of
$\overline{S_m\setminus S}$ together with the standard
triangulation of $S$. By projecting this into the cone over $S_m$
centered at the origin we get a fan $\delstar=(\sigmstar,\Cstar)$.
The simplicial set $\sigmstar$ can be described in the form
\[ \sigmstar=\Sigma\cup\{J*(J_1\cdots J_k)\mid J\subset J_1\subsetneqq
\ldots \subsetneqq J_k \subsetneqq I\}. \]
The case $(J_1\cdots J_k)=\emptyset$ is included. The simplex
$(J_1\cdots J_k)$
will be denoted by $K$ for simplicity. Then part of
$\sigmstar$ with $K=\emptyset$
is the subcomplex of $\Sigma$ corresponding to the boundary of $I$.
Part of $\sigmstar$ with $J=\emptyset$ will be denoted by $\partial\sigmstar$.
It is the set of sequences $(J_1\cdots J_k)$ such that
$J_1\subsetneqq \ldots \subsetneqq J_k\subsetneqq I$.
The cone $\Cstar(J*(J_1\cdots J_k))$ is the one generated by
$S_J*(b_{J_1}\ldots b_{J_k})$.

The cone $C(I)$ is triangulated by $\{\Cpr(\Jpr)\}$ with $\rho(\Jpr)\subset I$
and the simplex $S$ is triangulated accordingly.
If we replace the standard triangulation of $S$ by this triangulation
in the construction of $\delstar$ above and $S_J*(b_{J_1}\ldots b_{J_k})$
by $S_{\Jpr}*(b_{J_1}\ldots b_{J_k})$ with $\rho(\Jpr)\subset J_1$,
then we obtain another fan $\delprstar=(\sigmprstar,\Cprstar)$ where
\[ \sigmprstar=\sigmpr\cup\{\Jpr*(J_1\cdots J_k)\mid
\rho(\Jpr)\subset
J_1\subsetneqq \ldots\subsetneqq J_k \subsetneqq I\}\}. \]
Part of $\sigmprstar$ with $\Jpr=\emptyset$ will be denoted by
$\partial\sigmprstar$. It is identified with $\partial\sigmstar$ in an
obvious way.

We further assume that all the $b_J\in S_m$ are integral points.
This is
possible by choosing $m$ suitably. Moreover such $m$ can be taken
arbitrarily large. We put $v_J=b_J$. Then we define $\Vstar$ to be the sum
$\V\sqcup \{v_J\}$ where $(J)$ ranges over $(\partial\sigmstar)^{(1)}
\subset \sigmstar^{(1)}$.
Putting $u^I(\xi)=\sum_{i\in I}d_i\u$ define the number $d_J$ by
\begin{equation}\label{eq:dJ}
 d_J=\l u^I(\xi), v_J\r,
\end{equation}
and the $\Q$-divisor $\xistar$ by
\[ \xistar=\sum_{i\in \sigmone}d_ix_i+\sum_{(J)\in b\sigmone}d_Jx_J, \]
where $x_J$ is the basis element in $H_T^2(\delstar,\Vstar)$
corresponding to $(J)\in (\partial\sigmstar)^{(1)}$.

Similarly we define
\[ \Vprstar:=\Vpr\sqcup \{v_J\}_{(J)\in b\sigmone},\quad \xiprstar:=
 \sum_{\ipr\in\sigmprone}d_{\ipr}x_{\ipr}+\sum_{(J)\in b\sigmone}d_Jx_J. \]

The map $\rho:\delpr \to \Delta$ induces a map
$\rho:\delprstar \to \delstar$ by
\[ \rho(\Jpr*(J_1\cdots J_k))=\rho(\Jpr)*(J_1\cdots J_k)). \]
Then it is clear that
\[ \rho^*(\xistar)=\xiprstar. \]
We put
\[ b_{J*(J_1\cdots J_k)}(\delprstar,\Vprstar,\xiprstar)=
 \sum_{\Jpr;\rho(\Jpr)=J}b_{\Jpr*(J_1\cdots J_k)}
(\delprstar,\Vprstar,\xiprstar),\]
and
\[b_{J*(J_1\cdots J_k)}(\delprstar,\Vprstar,\xiprstar;
\delstar,\Vstar,\xistar)=b_{J*(J_1\cdots J_k)}(\delprstar,\Vprstar,\xiprstar)
-b_{J*(J_1\cdots J_k)}(\delstar,\Vstar,\xistar) \]
as before.
We are ready now for the statement $C_n$:
\vspace*{0.3cm}\\
$C_n$ \hspace{1cm}
$b_{J*(J_1\cdots J_k)}(\delprstar,\Vprstar,\xiprstar;
\delstar,\Vstar,\xistar)(u)\longrightarrow 0$
as $m\longrightarrow +\infty$  \\
\hspace*{1.6cm} for $\dim L=n$ and for all $J*(J_1\cdots J_k)$ with
$(J_1\cdots J_k)\not=\emptyset$.

\begin{lemm}\label{lemm:Cn}
The statement $C_n$ holds for $n\geq 2$.
\end{lemm}
\begin{proof}
We assume $B_k$ for $k<n$ and deduce $C_n$. The simplex
$(J_1\cdots J_k)$ will be denoted by $K$ as before. Its vertices are
written as $(J_1),\ldots ,(J_k)$ to distinguish from the simplices
$J_1, \ldots, J_k$ in $\Sigma$. In case $\dim J_k<n-2$ the dimension of
$J*K$ is less than $n-1$. Hence we can apply the inductive
assumption for $\rho:\sigmpr(J*K)\to \Sigma(J*K)$ where
$\sigmpr(J*K)=\{\Jpr_*\in \sigmprstar \mid
\rho(\Jpr_*)\subset J*K\}$ and $\Sigma(J*K)=
\{J_*\mid J_*\subset J*K\}$. It follows that
\[ b_{J*K}(\delprstar,\Vprstar,\xiprstar;\delstar,\Vstar,\xistar)=0. \]

Suppose next that $\dim J_k=n-2$, i.e. $J_k\in \Sigma^{(n-1)}$. Put
\[ H_{J*K}^0=\{h\in H_{J*K}\mid f_{h,(J_k)}=0\}.   \]
It is a subgroup of $H_{J*K}$. Similarly we put
\[ H_{\Jpr*K}^0=\{h\in H_{\Jpr*K}\mid f_{h,(J_k)}=0\}.   \]
Write
\[ b_{J*K}(\delprstar,\Vprstar,\xiprstar;\delstar,\Vstar,\xistar)(u)=b_0+b_1, \]
where
\[ \begin{split}b_0=&\Biggl(\sum_{\Jpr;\rho(\Jpr)=J}
 \sum_{h\in H_{\Jpr*K}^0}(-1)^{n-|\Jpr*K|}
 \prod_{\jpr\in \Jpr}{\textstyle\frac{(\zeta^{d_{\jpr}}q^{\l u,v_{\jpr}\r})^{f_{h,\jpr}}}
 {1- \zeta^{d_{\jpr}}q^{\l u,v_{\jpr}\r}}}
 \prod_{i<k}{\textstyle\frac{(\zeta^{d_{J_i}}q^{\l u,v_{J_i}\r})^{f_{h,(J_i)}}}
 {1- \zeta^{d_{J_i}}q^{\l u,v_{J_i}\r}} }\\
 &-\sum_{h\in H_{J*K}^0}(-1)^{n-|J*K|}
 \prod_{j\in J}{\textstyle\frac{(\zeta^{d_j}q^{\l u,v_j\r})^{f_{h,j}}}
 {1- \zeta^{d_j}q^{\l u,v_j\r}}}
 \prod_{i<k}{\textstyle\frac{(\zeta^{d_{J_i}}q^{\l u,v_{J_i}\r})^{f_{h,(J_i)}}}
 {1- \zeta^{d_{J_i}}q^{\l u,v_{J_i}\r}}}\Biggr){\textstyle
 \frac{1}{1-\zeta^{d_{J_k}}q^{\l u,v_{J_k}\r}}},
 \end{split}
\]
and
\[ \begin{split} &b_1=\Biggl(\sum_{\Jpr;\rho(\Jpr)=J}
 \sum_{h\not\in H_{\Jpr*K}^0}(-1)^{n-|\Jpr*K|}
 \prod_{\jpr\in \Jpr}{\textstyle\frac{(\zeta^{d_{\jpr}}q^{\l u,v_{\jpr}\r})^{f_{h,\jpr}}}
 {1- \zeta^{d_{\jpr}}q^{\l u,v_{\jpr}\r}}}
 \prod_{i<k}{\textstyle\frac{(\zeta^{d_{J_i}}q^{\l u,v_{J_i}\r})^{f_{h,(J_i)}}}
 {1- \zeta^{d_{J_i}}q^{\l u,v_{J_i}\r}} }\\
 & -\sum_{h\not\in H_{J*K}^0}(-1)^{n-|J*K|}
 \prod_{j\in J}{\textstyle\frac{(\zeta^{d_j}q^{\l u,v_j\r})^{f_{h,j}}}
 {1- \zeta^{d_j}q^{\l u,v_j\r}} }
 \prod_{i<k}{\textstyle\frac{(\zeta^{d_{J_i}}q^{\l u,v_{J_i}\r})^{f_{h,(J_i)}}}
 {1- \zeta^{d_{J_i}}q^{\l u,v_{J_i}\r}} }\Biggr)
 {\textstyle\frac{(\zeta^{d_{J_k}}q^{\l u,v_{J_k}\r})^{f_{h,(J_k)}}}
 {1-\zeta^{d_{J_k}}q^{\l u,v_{J_k}\r}} }.
 \end{split}
\]

The subgroup $H_{\Jpr*K}^0$ coincides with
$H_{\Jpr*K^0}$ where $K^0=(J_1\cdots J_{k-1})$. It follows that
the term inside the parenthesis in the expression for $b_0$ is nothing
but $b_{J*K^0}(\delprstar,\Vprstar,\xiprstar;\delstar,\Vstar,\xistar)$
and is equal to zero by induction assumption.
Thus $b_0=0$.

In order to estimate $b_1$ we first remark the following
\begin{lemm}\label{lemm:|H|}
The orders of $H_{J*K}=H_{J*K,\V_*}$ for $J\in \Sigma$ and
$H_{\Jpr*K}=H_{\Jpr*K,\V_*}$ for $\Jpr\in\sigmpr$
are bounded by $Cm^n$ for some constant $C$ depending only on
$\Delta, W$ and $v_0$.
\end{lemm}
Proof of Lemma. If $J*K$ is contained in $J_1*K_1$ then $|H_{J*K}|\leq
|H_{J_1*K_1}|$. So we may assume $J*K\in \Sigma_*^{(n)}$. Then
$H_{J*K}=L/L_{J*K,\V_*}$.
If we choose a basis of $L$ we can express $|H_{J*K}|$ as the determinant of
a matrix $A$ of degree $n$. Note that $v_{J_i}=v_0+m(\bar{v}_i-v_0)$ where
$\bar{v}_i=C(i)\cap W$. It follows that components of $A$ are linear
functions of $m$ with coefficients depending only on $\Delta, W$ and $v_0$.
Hence there exists a constant $C$ such that $|H_{J*K}|\leq Cm^n$.
Similarly $|H_{\Jpr*K}|\leq Cm^n$ for some constant $C$.
This proves Lemma \ref{lemm:|H|}.

\begin{lemm}\label{lemm:infty}
Let $g(x)$ be a polynomial in $x$.
The function $\frac{g(x)q^{fx}}{1-\zeta^dq^x}$ in real variable $x$
tends to $0$ as $x$ tends to $\pm\infty$ provided $0<f<1$, where $d\in \Q$.
It is bounded for $f=0$ and $g(x)=\text{constant}$.
\end{lemm}
In fact, recall that $\tau=a+b\img$ with $b>0$. Then
\[ \frac{g(x)q^{fx}}{1-\zeta^dq^x}=\frac{g(x)\alpha_1(x)e^{-2\pi bfx}}
{1-\zeta^d\alpha_2(x)e^{-2\pi bx}} \]
with $|\alpha_i(x)|=1,\ i=1,2$.
If $0<f<1$ then the right-hand side tends to $0$ as $x$ tends to $\pm\infty$.
If $f=0$ and $g(x)=\text{constant}$ it is bounded. This finishes the proof
of Lemma \ref{lemm:infty}.

In view of Lemma \ref{lemm:|H|} and Lemma \ref{lemm:infty} the absolute value
of the term inside the parenthesis in the expression for $b_1$ is bounded by
$\Cpr m^n$ with some constant $\Cpr$. Hence
\[ |b_1|\leq \Cpr m^n
|\frac{(\zeta^{d_{J_k}}q^{\l u,v_{J_k}\r})^{f_{h,(J_k)}}}
 {1-\zeta^{d_{J_k}}q^{\l u,v_{J_k}\r}}|. \]
$\l u,v_{J_k}\r$ is a linear function of $m$ and its absolute value
tends to $\infty$ when $m$ tends to $\infty$.
Then $b_1$ and hence $b_{J*K}(\delprstar,\Vprstar,\xiprstar;
\delstar,\Vstar,\xistar)(u)$ tend to $0$ by Lemma \ref{lemm:infty}.
This finishes the proof of Lemma \ref{lemm:Cn}.
\end{proof}

We now prove that $C_n$ implies $B_n$ assuming $B_k$ for $k<n$. For that
purpose we introduce
the following multi-fan $\delstartil=(\sigmstartil,\Cstartil,w^\pm)$.
$\sigmstartil$ is the union of $\sigmprstar$ and $\sigmstar$ glued
along $\partial\sigmstar=\partial\Sigma'_*$, and
$\Cstartil\vert\sigmprstar=\Cprstar$ and
$\Cstartil\vert\sigmstar=\Cstar$.

We define the function $w^\pm$ for $\sigmstartil$ as follows:
\begin{align*}
 w^+(J_*)=\begin{cases}  1, & J_*\in (\sigmprstar)^{(n)}, \\
           0,  & J_*\in \sigmstar^{(n)}, \end{cases}
\quad
 w^-(J_*)=\begin{cases} 0, &J\in (\sigmprstar)^{(n)}, \\
          1,    &J_*\in \sigmstar^{(n)}. \end{cases}
\end{align*}
Then
\begin{align*}
   w(J_*)=\begin{cases} 1, &  \text{for $J_*\in(\sigmprstar)^{(n)}$,} \\
           -1, &  \text{for $J_*\in\sigmstar^{(n)}$.} \end{cases}
\end{align*}
The triple $(\sigmstartil,\Cstartil,w^\pm)$ defines a simplicial multi-fan
$\delstartil$.

It can be proved in a similar way to \eqref{eq:startil}
that $\delstartil$ is complete and the following equalities hold.
\begin{equation}\label{eq:startil2}
\deg((\delstartil)_{J_*})=\begin{cases}
      +1, & J_*\in \sigmprstar\setminus \partial\sigmstar, \\
      -1, & J_*\in \sigmstar\setminus \partial\sigmstar, \\
      \ \ 0, & J_*\in \partial\sigmstar.
 \end{cases}
\end{equation}

$\Vstar$ on $\sigmstar$ and $\Vprstar$ on $\sigmprstar$ define
$\Vstartil$ on $\sigmstartil$. Also $\xistar$ on $\sigmstar$ and
$\xiprstar$ on $\sigmprstar$ define $\xistartil$ on $\sigmstartil$.
Put $u=u^I(\xi)=\sum_{i\in I}d_i\u$ as before. We claim that
\[ \xistartil=u \in H_T^2(\delstartil,\Vstartil)\otimes \Q. \]
In fact $\l u^I(\xi),v_i\r =d_i$ by definition of $u^I(\xi)$ and
$\l u^I(\xi), v_{\ipr}\r=d_{\ipr}$ by \eqref{eq:dipr2}.
Similarly $\l u^I(\xi), v_{J}\r=d_{J}$ by \eqref{eq:dJ}. These
equalities imply $\xistartil=u^I(\xi)$.

We apply Theorem \ref{theo:varrigid} and get
\begin{equation}\label{eq:hatvarst0}
 \hatvarst(\delstartil,\Cstartil,\xistartil)=0.
\end{equation}
On the other hand by using \eqref{eq:BL2} and \eqref{eq:startil2}
we have
\begin{align*}\label{eq:hatvarst*}
 &\hatvarst(\delstartil,\Vstartil,\xistartil) \\
      &=b_I(\delpr,\Vpr,\xipr;\Delta,\V,\xi)+
 \sum_{J*K\in \sigmstar\setminus \sigmn}
 b_{J*K}(\delprstar,\Vprstar,\xiprstar;\delstar,\Vstar,\xistar).
\end{align*}
For $K=\emptyset$ the term
$b_{J*K}(\delprstar,\Vprstar,\xiprstar;\delstar,\Vstar,\xistar)
=b_J(\delpr,\Vpr,\xipr;\Delta,\V,\xi)=0$ by induction assumption.
In case $K\not=\emptyset$,
$b_{J*K}(\delprstar,\Vprstar,\xiprstar;\delstar,\Vstar,\xistar)(u)$
tends to $0$ when $m$ tends to $\infty$ by Lemma \ref{lemm:Cn}.
From this and \eqref{eq:hatvarst0} it follows that
$b_I(\delpr,\Vpr,\xipr;\Delta,\V,\xi)$ must be equal to zero.
Together with inductive assumption this proves that $B_n$ holds.

Thus Theorem \ref{thm:Bn} is proved.

The rest of this section is devoted to the proof of Theorem
\ref{thm:crepant1}.
We shall first show that $\delpr$ is complete and the following equality
holds for every $\Jpr\in \sigmpr$:
\begin{equation}\label{eq:deg}
 \deg(\delpr_{\Jpr})=\deg(\Delta_{\rho(\Jpr)}).
\end{equation}

We use the notation $V=L_\R$ and $V_K=(L_K)_\R$ where $(L_K)_\R$ is
as in Section 2. Similar notations are used for $K'\in \sigmpr$.
Let $v$ be a generic vector in $V/V_K$ and put
\[ S_v(K)=\{I\in \Sigma_K^{(n-k)}\mid v\in C_K(I)\} \]
where $\Sigma_K^{(n-k)}=\{I\in \sigmn\mid K\subset I\}$ and $C_K(I)$
is the image of $C(I)$ in $V/V_K$.
Recall that $\deg(\Delta_K)$ is defined to be equal to
\[ \sum_{I\in S_v(K)}w(I). \]
Completeness of $\Delta$ implies that this is independent of
generic vector $v$. We take a generic vector $v$ in $V$ and denote
its image in $V/V_K$ also by $v$. In this sense $v$ may be
considered as a vector in $V/V_{\Jpr}$ and in $V/V_{\rho(\Jpr)}$
at the same time. Suppose that $\Jpr$ lies in
$\Sigma^{\prime(l')}$ and $\rho(\Jpr)$ in $\Sigma^{(l)}$.

We shall show
\begin{equation*}\label{eq:deg2}
 \sum_{\Ipr\in S_v(\Jpr)}w(\Ipr)=
 \sum_{I\in S_v(\rho(\Jpr))}w(I).
\end{equation*}
Since the right-hand side is independent of $v$ (being equal to
$\deg(\Delta_{\rho(\Jpr)})$) this would imply that $\delpr$ is
complete and the equality \eqref{eq:deg} holds.
Since $w(\Ipr)=w(\rho(\Ipr))$ it is enough to show that
there is a bijection $S_v(\Jpr)\to S_v(\rho(\Jpr))$.

The projection $p: V/V_{\Jpr} \to V/V_{\rho(\Jpr)}$ maps every cone
$\Cpr_{\Jpr}(\Ipr)$ for $\Ipr\in \Sigma_{\Jpr}^{\prime (n-l')}$ to a cone
contained in $C_{\rho(\Jpr)}(\rho(\Ipr))$. In particular $p$ defines
a map $S_v(\Jpr)\to S_v(\rho(\Jpr))$.
Take a simplex $I\in S_v(\rho(\Jpr))$,
and put $\sigmpr(I)=\{\Ipr\in \Sigma_{\Jpr}^{\prime(n-l')}\mid
\rho(\Ipr)=I\}$.
Then the cones $\{\Cpr(\Ipr)\}$ with $\Ipr\in \sigmpr(I)$ are contained
in the image of $C(I)$ in $V/V_J$, and together with their faces, they
form a fan in $V/V_{\Jpr}$. Since $C(I)$ contains $v$ regarded as a point in
$V/V_{\rho(\Jpr)}$, there is exactly one simplex $\Ipr\in \sigmpr(I)$
such that $\Cpr(\Ipr)$ contains $v$ regarded as a point in $V/V_{\Jpr}$,
that is, there is one and only one simplex
$\Ipr\in S_v(\Jpr)$. This shows that the map
$S_v(\Jpr)\to S_v(\rho(\Jpr))$ is a bijection and proves
\eqref{eq:deg}.

Using \eqref{eq:deg} and Theorem \ref{thm:Bn} we have
\begin{align*}
 \hatvarst(\delpr,\Vpr,\xipr)&=\sum_{\Jpr\in \sigmpr}
  \deg(\delpr_{\Jpr})b_{\Jpr}(\delpr,\Vpr,\xipr)  \\
  &=\sum_{J\in \Sigma}\sum_{\rho(\Jpr)=J}
  \deg(\delpr_{\Jpr})b_{\Jpr}(\delpr,\Vpr,\xipr)  \\
  &=\sum_{J\in \Sigma}\deg(\Delta_J)b_J(\Delta,\V,\xi) \\
  &=\hatvarst(\Delta,\V,\xi).
\end{align*}
This finishes the proof of Theorem \ref{thm:crepant1}.

\begin{rema}\label{rem:crepant}
Let $\Delta$ be a fan such that every $J\in \Sigma$ is contained in some
$I\in \sigmn$. As was pointed out in Proposition \ref{prop:BL2}
the equality
\[ \hatvarvst(\Delta,\V,\xi)=\sum_{u\in L^*}t^{-\l u,v\r}
\sum_{J\in \Sigma}\deg(J)b_J(\Delta,\V,\xi)(u) \]
holds for $v\in \bigcup_{I\in \sigmn}C(I)$. If $\rho:\delpr\to \Delta$
is a map satisfying a), b), then $\delpr$ is a fan and
$\hatvarvst(\delpr,\Vpr,\xipr)$ has a
meaning for $v\in \bigcup_{I\in \sigmn}C(I)$. Theorem \ref{thm:Bn}
implies
\begin{equation*}\label{eq:vcrepant}
\hatvarvst(\delpr,\Vpr,\xipr)=\hatvarvst(\Delta,\V,\xi)
\end{equation*}
and
\begin{equation*}\label{eq:crepant}
\hatvarst(\delpr,\Vpr,\xipr)=\hatvarst(\Delta,\V,\xi),
\end{equation*}
if $\hatvarst(\delpr,\Vpr,\xipr)$ and $\hatvarst(\Delta,\V,\xi)$ are
defined as in Remark \ref{rem:BL}.
\end{rema}

\section{Invariance of orbifold elliptic class under push-forward $\rho_*$}

Let $\rho:\delpr \to \Delta$ be a map satisfying a), b).
We shall define a functorial map
\[ \rho_*:S^{-1}H_T^*(\delpr,\Vpr)\otimes\Q\to
 S^{-1}H_T^*(\Delta,\V)\otimes\Q. \]
 The following equality will hold if
every simplex $J\in \Sigma$ is contained in some $J\in \sigmn$,
cf. Theorem \ref{thm:rhoclass}.
\[ \rho_*(\hatEst(\delpr,\Vpr,\xipr))=\hatEst(\Delta,\V,\xi)
 \quad \text{for $\xipr=\rho^*(\xi)$}.
\]

In order to define $\rho_*(x)$ for
$x\in H_T^*(\Delta,\V)\otimes\Q$ it is sufficient to determine
$\rho_*(x)_I=\iota_I^*(\rho_*(x))\in S^{-1}S^*(L_{I,\V})\otimes\Q
=S^{-1}H^*(BT)\otimes\Q$ for
each $I\in \sigmn$ satisfying
\begin{equation}\label{eq:rhoI}
 \iota_{I_1\cap I_2}^{I_1}(\rho_*(x)_{I_1})=
 \iota_{I_1\cap I_2}^{I_2}(\rho_*(x)_{I_2}) \quad
 \text{for any $I_1, I_2\in \sigmn$ with $I_1\cap I_2\not=\emptyset$}
\end{equation}
in view of \eqref{eq:pistarimage}.
Note that the localized ring $S^{-1}S^*(L)\otimes\Q$ is nothing but the
algebra of rational functions on $L_\Q$.
Put $u_I=\prod_{i\in I}\u$. We then define
\begin{equation*}\label{eq:rhoI2}
 \rho_*(x)_I=|H_I|u_I\sum_{\Ipr\in \sigmprn,\rho(\Ipr)=I}
 \frac{\iota_{\Ipr}^*(x)}{|H_{\Ipr}|u_{\Ipr}}.
\end{equation*}

In order to show that the $\rho_*(x)_I$ in fact satisfy \eqref{eq:rhoI},
we introduce logarithmic forms $\OmeI$ and $\omeI$ on $V=L_\R$ for
each $I\in \sigmn$. Give an orientation to $I$ and let
$I=\{i_1,\ldots,i_n\}$ be the ordering of $I$ concordant to
the orientation. Put $w_{i_\nu}^I=\frac{du_{i_\nu}^I}{u_{i_\nu}^I}$
and
\[
 \OmeI=w_{i_1}^I\wedge \cdots \wedge w_{i_n}^I,\quad
 \omeI=\sum_{\nu=1}^n(-1)^{\nu-1}w_{i_1}^I\wedge\cdots\wedge
 \widehat{w_{i_\nu}^I}
   \wedge\cdots\wedge w_{i_n}^I
\]
where \ $\hat{}$ \ means to delete the underlying symbol.

\begin{lemm}\label{lemm:OomI}
Give the concordant orientation with $I$ to $\Ipr$ such that $\rho(\Ipr)=I$.
Then
\begin{equation}\label{eq:OmI}
 \sum_{\Ipr\in \sigmprn,\rho(\Ipr)=I}
 \iota_{\Ipr}^*(x)\OmeIpr=\rho_*(x)_I\OmeI,
\end{equation}
and
\begin{equation}\label{eq:omI}
 \sum_{\Ipr\in \sigmprn,\rho(\Ipr)=I}
 \iota_{\Ipr}^*(x)\omeIpr=\rho_*(x)_I\omeI.
\end{equation}
\end{lemm}

\begin{proof}
For simplicity we put $I=\{1,2,\ldots,n\}$ with
orientation determined by this ordering. The ordering
$(v_1,v_2,\ldots ,v_n)$ gives an orientation to the vector space $V=L_\R$.
Take an ordered integral basis of the lattice $L$ which is concordant
with the orientation of $V=L_\R$ and let $u_1,\ldots,u_n$ be the
corresponding coordinates. Put $\Theta=du_1\wedge \cdots \wedge du_n$. Then
\[ du_1^I\wedge \cdots \wedge du_n^I =
 \frac{1}{|H_I|}du_1\wedge \cdots \wedge du_n=\frac{1}{|H_I|}\Theta. \]
Hence
\[ \OmeI=\frac{\Theta}{|H_I|u_I}. \]
Similarly
\[ \OmeIpr=\frac{\Theta}{|H_{\Ipr}|u_{\Ipr}}. \]
Then
\[
 \sum_{\rho(\Ipr)=I}\iota_{\Ipr}^*(x)\OmeIpr
 =\sum_{\rho(\Ipr)=I}\frac{\iota_{\Ipr}^*(x)}{|H_{\Ipr}|u_{\Ipr}}\Theta
 = \rho_*(x)_I\frac{\Theta}{|H_I|u_I}=\rho_*(x)_I\OmeI.
\]

Put $\theta=\sum_{i=1}^n(-1)^{i-1}u_idu_1\wedge\cdots\wedge
 \widehat{du_i}\wedge\cdots\wedge du_n$. Then we also have
\[ \omeI=\frac{\theta}{|H_I|u_I}. \]
From this \eqref{eq:omI} follows
in an entirely similar way.
\end{proof}

Fix $I\in \sigmn$ and $i\in I$ and put
$J=I\setminus \{i\} \in \Sigma^{(n-1)}$.
Let $\Sigma(J)$ be the simplicial set consisiting of all faces of $J$ and
let $\sigmpr(J)=\{\Jpr\in \sigmpr\mid \rho(\Jpr)\subset J\}$.
Fans $\Delta(J)$ and $\delpr(J)$ are induced from $\Delta$ and
$\delpr$ by restriction on $\Sigma(J)$ and $\sigmpr(J)$ respectively.
Put also
\[ \V_J=\{v_j\mid j\in \Sigma(J)^{(1)}\}\ \ \text{and}\ \
  \Vpr_J=\{v_{\jpr}\mid \jpr\in \sigmpr(J)^{(1)}\}. \]
Then $\rho$ induces $\rho|J:(\sigmpr(J),\Vpr_J)\to
 (\Sigma(J),\V_J)$ satisfying a),b).

\begin{lemm}\label{lemm:rhoJ}
\[ ((\rho|J)_*(x|J))_J=(\rho_*(x)_I)|J. \]
In other words
\begin{equation}\label{eq:rhoJ}
\sum_{\Jpr\in\sigmpr(J)^{(n-1)}}
 \iota_{\Jpr}^*(x|J)\Omega^{\Jpr}=(\rho_*(x)_I)|J\cdot\Omega^J.
\end{equation}
as logarithmic forms in the hyperplane $(L_J)_\R$ containing $C(J)$.
Here $x|J$ stands for the image of $x$ by the map
$H_T^*(\delpr,\Vpr)\to H_T^*(\delpr(J),\Vpr_J)$ sending
$x_{\jpr}$ to $x_{\jpr}$ for $\jpr\in \sigmpr(J)^{(1)}$ and to $0$ for
$\ipr\not\in \sigmpr(J)^{(1)}$. Also $u|J=\iota_J^{I*}(u)$ for
$u\in S^{-1}S^*(L_{I,\V})$.
\end{lemm}

\begin{proof}
In general let $V_1,\ldots,V_n$ be the hyperplanes spanned by
$(n-1)$-dimensional
faces of a strongly convex $n$-dimensional simplicial cone $C$ in
an $n$-dimensional vector space $V$. Let $w_i=\frac{du_i}{u_i}$ be the
logarithmic $1$-form corresponding to $V_i$. Here $u_i$ is a linear form
vanishing on $V_i$. Note that $w_i$ depends only on $V_i$ but not
on particular $u_i$. Put
\[ \omega_i=(-1)^{i-1}w_1\wedge\cdots\wedge
 \widehat{w_i}\wedge\cdots\wedge w_n \quad \text{and} \quad
 \omega=\sum_{i=1}^n\omega_i. \]

\begin{sub}\label{sub:omegavanish}
Let $V_0$ be a hyperplane defined by $\sum_ia_iu_i=0$. If $V_0$
is different from $V_i,\ 1\leq i\leq n$, then $\omega|V_0=0$.
\end{sub}
Since none of the $V_i$ coincides with $V_0$ there are at least two
non-zero $a_i$. We may suppose that $a_n\not=0$. Then
$u_n=-\frac{\sum_{i=1}^{n-1}a_iu_i}{a_n}$ on $V_0$ and
\begin{align*}
 \omega_i |V_0&=(-1)^{i-1}w_1\wedge\cdots\wedge
 \widehat{w_i}\wedge\cdots\wedge w_{n-1}\wedge
 \frac{-d(\sum_{i=1}^{n-1}a_iu_i)}{a_nu_n} \\
 &=(-1)^{n+1}\frac{a_idu_1\wedge\cdots\wedge du_{n-1}}
 {a_n\prod_{j\not=i}u_j}.
\end{align*}
for $i\not=n$. The last equality automatically holds for $i=n$.
Therefore
\begin{align*} \omega|V_0&=\sum_{i=1}^n\omega_i|V_0 \\
  &=(-1)^{n+1}\frac{1}{a_n}\sum_{i=1}^n\frac{a_i}{\prod_{j\not=i}u_j}
  du_1\wedge\cdots\wedge du_{n-1} \\
  &=(-1)^{n+1}\frac{1}{a_n\prod_{j=1}^nu_j}(\sum_{i=1}^na_iu_i)
  du_1\wedge\cdots\wedge du_{n-1}=0
\end{align*}
Thus Sublemma \ref{sub:omegavanish} holds.

We continue with the proof of Lemma \ref{lemm:rhoJ}. We may suppose
without loss of generality that $I=\{1,\ldots,n\}$ and $i=n$ and
$J=I\setminus \{n\}$. Put $V_0=(L_J)_\R$ and
$\omega_i^I=(-1)^{i-1}w_1^I \wedge\cdots\wedge
\widehat{w_i^I} \wedge\cdots\wedge w_n^I$.
We also use similar notations for $\Ipr$.

Note that
\begin{equation}\label{eq:rhoomega}
 \rho_*(x)_I\omega_n^I|V_0=(-1)^{n-1}(\rho_*(x)_I)|J\cdot\Omega^J,
\end{equation}
since $\omega_n^I|V_0=(-1)^{n-1}\Omega^J$.
We next consider the contribution from the left hand side
of \eqref{eq:omI} to the term
$\rho_*(x)_I\omega_n^I$ in $\rho_*(x)_I\omega^I=\sum_i\rho_*(x)_I\omega_i^I$.
Sublemma \ref{sub:omegavanish} implies that it suffices to consider
only such $\Ipr$ that $\rho(\Ipr)=I$ and that some facet cone $\Cpr(\Jpr)$
of $\Cpr(\Ipr)$ is contained in $C(J)\subset V_0$. The set of
such $\Ipr$ will be
denoted by $\sigmpr(I,J)^{(n)}$. Let $\Jpr=\Ipr\setminus \{n'(\Ipr)\}$.
In this case $\omega_{\ipr}^{\Ipr}$ for $\ipr\not=n'(\Ipr)$ containes
$w_n^I=\frac{du_n^I}{u_n^I}$
when expressed in terms of $\omega_i^I$. Hence the contribution to
$\rho_*(x)_I\omega_n^I$ comes only from
$\iota_{\Ipr}^*(x)\omega_{n'(\Ipr)}^{\Ipr}$.
It follows that
\[ \rho_*(x)_I\omega_n^I|V_0
 =\sum_{\Ipr\in \sigmpr(I,J)^{(n)}}
 \iota_{\Ipr}^*(x)\omega_{n'(\Ipr)}^{\Ipr}|V_0
 =\sum_{\Jpr\in\sigmpr(J)^{(n-1)}}
 \iota_{\Jpr}^*(x|J)\cdot\omega_{n'(\Ipr)}^{\Ipr}|V_0. \]
Combining this with \eqref{eq:rhoomega} and noting that
$\omega_{n'(\Ipr)}^{\Ipr}|V_0=(-1)^{n-1}\Omega^{\Jpr}$
we obtain \eqref{eq:rhoJ}. This finishes the proof of Lemma \ref{lemm:rhoJ}.
\end{proof}

Lemma \ref{lemm:rhoJ} shows that that $(\rho_*(x)_I)|J$ depends only on
$J$. \eqref{eq:rhoI} follows from this when $I_1\cap I_2$ has codimension $1$.
The general case is proved by induction on codimensions using Lemma \ref{lemm:rhoJ}. Thus
$\rho_*(x)$ is well-defined. Once this is established one can rewrite
Lemma \ref{lemm:OomI} in the following
\begin{prop}\label{prop:OmI}
\begin{equation*}\label{eq*:OmI}
 \sum_{\Ipr\in \sigmprn,\rho(\Ipr)=I}
 \iota_{\Ipr}^*(x)\OmeIpr=\iota_I^*(\rho_*(x))\OmeI,
\end{equation*}
and
\begin{equation*}\label{eq*:omI}
 \sum_{\Ipr\in \sigmprn,\rho(\Ipr)=I}
 \iota_{\Ipr}^*(x)\omeIpr=\iota_I^*(\rho_*(x))\omeI.
\end{equation*}
\end{prop}

\begin{prop}
$\rho_*$ maps $H_T^*(\delpr,\Vpr)\otimes\Q$ into
$H_T^*(\Delta,\V)\otimes\Q$.
\end{prop}
\begin{proof} Take $x\in H_T^*(\delpr,\Vpr)\otimes\Q$. It is enough
to show that $\iota_I^*(\rho_*(x))$ belongs to
$S^*(L)\otimes\R$ for any $I\in \sigmn$. Note that
$\OmeI$ has simple poles only along $V_J=(L_J)_\R$ for
$J\in \Sigma^{(n-1)}$ with $J\subset I$. In view of Proposition
\ref{prop:OmI} it suffices to show that
$\Omega'=\sum_{\Ipr\in \sigmprn,\rho(\Ipr)=I} \iota_{\Ipr}^*(x)\OmeIpr$
has at most simple poles along the same loci $\{V_J\}$ as $\OmeI$.

Poles of $\Omega'$ appear along hyperplanes $(L_{\Jpr})_\R$ where $\Jpr$
runs over $(\sigmpr)^{(n-1)}$ with $\rho(\Jpr)\subset I$.
If $\Jpr$ is such that $\rho(\Jpr)=I$, then
there are exactly two $\Ipr_1,\Ipr_2\in \sigmprn$ that have $\Jpr$ as common
face. If $\Jpr=\{\jpr_1,\ldots,\jpr_{n-1}\}$, $\Ipr_1=\Jpr\cup \{\ipr_1\}$
and $\Ipr_2=\Jpr\cup \{\ipr_2\}$, then $\omega_{\ipr_1}=
\frac{du_0}{u_0}=\omega_{\ipr_2}$, where $u_0$ is a linear form vanishing on
$(L_{\Jpr})_\R$. Since $\Ipr_1$ and $\Ipr_2$ induce opposite orientations
on $\Jpr$,
{\small
\[
\iota_{\Ipr_1}^*(x)(\Omega)^{\Ipr_1}+
 \iota_{\Ipr_2}^*(x)(\Omega)^{\Ipr_2}=
 \pm \frac{du_0}{u_0}\wedge\left(\iota_{\Ipr_1}^*(x)w_{\jpr_1}^{\Ipr_1}
 \wedge\cdots\wedge w_{\jpr_{n-1}}^{\Ipr_1}
 -\iota_{\Ipr_2}^*(x)w_{\jpr_1}^{\Ipr_2}\wedge\cdots\wedge
 w_{\jpr_{n-1}}^{\Ipr_2}\right). \]}
The restriction of the form $\omega'$ in the parenthesis to the
hyperplane $(L_{\Jpr})_\R$ vanishes. Hence
$\iota_{\Ipr_1}^*(x)(\Omega)^{\Ipr_1}+
 \iota_{\Ipr_2}^*(x)(\Omega)^{\Ipr_2}=\pm\frac{du_0}{u_0}\wedge\omega'$
has no pole along $(L_{\Jpr})_\R$. This implies that $\Omega'$ has no pole
along $(L_{\Jpr})_\R$ for $\Jpr$ such that $\rho(\Jpr)=I$.

For $\Jpr\in (\sigmpr)^{(n-1)}$
such that $\rho(\Jpr)$ is a facet $J$ of $I$,
$\Omega'$ has at most simple pole along $(L_{\Jpr})_\R=(L_J)_\R=V_J$.
We have proved that $\Omega'$ has at most simple poles along the same
loci as $\OmeI$.
Hence $\rho_*(x)_I$ has no pole and in fact it is a polynomial.
\end{proof}

Functorial properties of $\rho_*$ are expressed in the following
\begin{prop}\label{prop:rhostar}
$\rho_*:H_T^*(\delpr,\Vpr)\otimes \Q\to
 H_T^*(\Delta,\V)\otimes\Q$ satisfies the following properties:
\begin{enumerate}
\item
\begin{align*}
(\rho_2\circ\rho_1)_*&=(\rho_2)_*\circ(\rho_1)_*\ \text{for
$\rho_1:(\Delta'',V'')\to (\delpr,\Vpr)$ and
$\rho_2:(\delpr,\Vpr)\to (\Delta,\V)$}, \\
 (id)_*&=id \ \text{for the identity map $id:(\Delta,\V)\to(\Delta,\V)$}.
\end{align*}
\item
 $\rho_*(1)=1$ for $1\in H_T^0(\Delta)$.
\item
 $\rho_*(x\rho^*(y))=\rho_*(x)y$ for $y\in
 H_T^*(\Delta,\V)\otimes\Q$.
\item
$\rho_*$ is an $H^*(BT)\otimes\Q$-module map.
\item
Assume that $\rho$ satisfies the condition c) in addition to a) and b).
Then $\rho_*$ is compatible with the push-forward $\pi_*$ to a point,
i.e., the following diagram commutes:
\[ \begin{CD}
 H_T^*(\delpr,\Vpr)\otimes \Q @>{\rho_*}>> H_T^*(\Delta,\V)\otimes\Q \\
 @V{\pi_*}VV @VV{\pi_*}V \\
 S^{-1}H^*(BT)\otimes\Q @= S^{-1}H^*(BT)\otimes\Q.
 \end{CD} \]
\end{enumerate}
\end{prop}
\begin{proof}
By Proposition \ref{prop:OmI}
\begin{multline*}
 \iota_I^*(\rho_{2*}(\rho_{1*}(x)))\Omega^I=
\sum_{\rho_2(\Ipr)=I}\iota_{\Ipr}^*(\rho_{1*}(x))\Omega^{\Ipr}
=\sum_{\rho_2(\Ipr)=I}\sum_{\rho_1''(I'')=\Ipr}
\iota_{I''}^*(x)\Omega^{I''} \\
=\sum_{(\rho_2\circ \rho_1)(I'')=I}\iota_{I''}^*(x)\Omega^{I''}
=\iota_I^*((\rho_2\circ \rho_1)_*(x))\OmeI.
\end{multline*}
Thus $\rho_{2*}\circ\rho_{1*}=(\rho_2\circ \rho_1)_*$.
$(id_*)=id$ clearly holds.

Since $\iota_I^*(1)=1$, $\iota_I^*(\rho_*(1))\omeI=
\sum_{\rho(\Ipr)=I}\omeIpr=\omeI$. Thus $\rho_*(1)=1$.

In order to prove the equality $\rho_*(x\rho^*(y))=\rho_*(x)y$ we may
assume $y=x_i\in H_T^2(\Delta,\V)$ as is easily seen. Then
\[ \iota_I^*(\rho_*(x\rho^*(x_i)))\OmeI
 =\sum_{\rho(\Ipr)=I}\iota_{\Ipr}^*(x\rho^*(x_i))\OmeIpr
 =\sum_{\rho(\Ipr)=I}\iota_{\Ipr}^*(x)\iota_{\Ipr}^*(\rho^*(x_i))\OmeIpr
\]
But $\iota_{\Ipr}^*(\rho^*(x_i))=\iota_I^*(x_i)$ by \eqref{eq:iotarho}. Hence
\[ \iota_I^*(\rho_*(x\rho^*(x_i)))\OmeI
 =\iota_I^*(x_i)\sum_{\rho(\Ipr)=I}\iota_{\Ipr}^*(x)\OmeIpr
 =\iota_I^*(x_i)\iota_I^*(\rho_*(x))\OmeI
 =\iota_I^*(\rho_*(x)x_i)\OmeI. \]
Thus $\rho_*(x\rho^*(x_i))=\rho_*(x)x_i$.

Since $\rho^*(u)=u$ for $u\in L_\Q^*$ we have
\[ \rho_*(ux)=\rho_*(x\rho^*(u))=\rho_*(x)u. \]
This shows that $\rho_*$ is an $H_T^*(BT)\otimes\Q$-module map.

Finally, admitting the condition c),
\begin{align*}
 \pi_*(\rho_*(x))&=\sum_{I\in \sigmn}\frac{w(I)\iota_I^*(\rho_*(x))}
 {|H_I|u_I} \\
 &=\sum_I\frac{1}{|H_I|u_I}\sum_{\rho(\Ipr)=I}
 \frac{w'(I')\iota_{\Ipr}^*(x)|H_I|u_I}{|H_{\Ipr}|u_{\Ipr}}
 =\sum_{\Ipr\in\sigmprn}\frac{w'(I')\iota_{\Ipr}^*(x)}
 {|H_{\Ipr}|u_{\Ipr}}
 =\pi_*(x)
\end{align*}
\end{proof}

The map $\rho_*:H_T^*(\delpr,\Vpr)\otimes\Q\to
H_T^*(\Delta,\V)\otimes\Q$ extends in an obvious way to
$\rho_*:(H_T^*(\delpr,\Vpr)\otimes\C)[[q]]\to
(H_T^*(\Delta,\V)\otimes\C)[[q]]$.

\begin{theo}\label{thm:rhoclass}
Let $\Delta$ be a simplicial multi-fan in a lattice $L$ of dimension $n$.
Assume that every simplex $J\in \Sigma$ is contained in some $J\in \sigmn$.
Let $\rho:\delpr \to \Delta$ a map satisfying a) and b).
Let $\V=\{v_i\}_{i\in \sigmone}$ and $\Vpr=\{v_{\ipr}\}_{\ipr \in \sigmprone}$
be sets of edge vectors for $\Delta$ and
$\delpr$ respectively. Put $\xipr=\rho^*(\xi)$.
Then the following equality holds:
\[ \rho_*(\hatEst(\delpr,\Vpr,\xipr))=\hatEst(\Delta,\V,\xi). \]
\end{theo}
\begin{proof}
It is enough to show that
\begin{equation}\label{eq:iotahat}
 \iota_I^*(\rho_*(\hatEst(\delpr,\Vpr,\xipr)))
 =\iota_I^*(\hatEst(\Delta,\V,\xi))
\end{equation}
for any $I\in \sigmn$. Let $\Sigma(I)$ be the simplicial set consisting of
faces of $I$ and put $\sigmpr(I)=\{\Jpr\in \sigmpr\mid \rho(\Jpr)\subset I\}$.
Cone structures of $\Delta$ and $\delpr$ restricted on $\Sigma(I)$ and
$\sigmpr(I)$ define fans $\Delta(I)$ and
$\delpr(I)$ respectively. We forget $w$ and $w'$ in $\Delta$ and $\delpr$,
and put $w(I)=1$ and $w'(\Ipr)=1$ for all
$\Ipr\in \sigmpr(I)^{(n)}$. With this understanding,
$\rho:\sigmpr(I)\to\Sigma(I)$ satisfies the conditions a), b) and c).
Then
\begin{align*}
\iota_I^*(\hatEst(\Delta,\V,\xi))&=\iota_I^*(\hatEst(\Delta(I),\V|I,\xi|I)) \\
\iota_I^*(\rho_*(\hatEst(\delpr,\Vpr,\xipr)))&=
 \iota_I^*(\rho_*(\hatEst(\delpr(I),\Vpr|I,\xipr|I))).
\end{align*}
By definition
\[\pi_*(\hatEst(\Delta(I),\V|I,\xi|I))=\hatvarepst(\Delta(I),\V|I,\xi|I),\]
and by Proposition \ref{prop:rhostar}
\[ \pi_*(\rho_*(\hatEst(\delpr(I),\Vpr|I,\xipr|I)))
 =\pi_*(\hatEst(\delpr(I),\Vpr|I,\xipr|I))
 =\hatvarepst(\delpr(I),\Vpr|I,\xipr|I).
\]
On the other hand, by
Theorem \ref{thm:crepant1} and
Remark \ref{rem:crepant}, one has
\[ \hatvarst(\delpr(I),\Vpr|I,\xipr|I) =\hatvarst(\Delta(I),\V|I,\xi|I). \]
From this equality and Remark \ref{rem:chhatvar} one gets
\[\begin{split}
\hatvarepst(\delpr(I),\Vpr|I,\xipr|I)
&=ch(\hatvarst(\delpr(I),\Vpr|I,\xipr|I)) \\
&=ch(\hatvarst(\Delta(I),\V|I,\xi|I))=\hatvarepst(\Delta(I),\V|I,\xi|I).
\end{split} \]
It follows that
\[ \pi_*(\rho_*(\hatEst(\delpr(I),\Vpr|I,\xipr|I)))=
 \pi_*(\hatEst(\Delta(I),\V|I,\xi|I)). \]
Since $\iota_I^*(x)=|H_I|u_I\pi_*(x)$ on $\Delta(I)$
we get \eqref{eq:iotahat}.
\end{proof}

\section {Generalization to $\Q$-Cartier triples}

So far we dealt only with simplicial multi-fans.
A (not necessarily simplicial) multi-fan
in an $n$-dimensional lattice $L$ is a triple
$\Delta=(\Sigma,C,w^\pm)$ where $\Sigma$ is a partially ordered
set with a unique minimum element $*$. We denote the partial
ordering by $\preceq$. $C$ is a map from $\Sigma$ to the set of
strongly convex rational cones in $L_\R$ satisfying the following
three conditions:
\begin{enumerate}
\item $C(*)=\{0\}$;
\item If $K\preceq J$ for $K,J\in \Sigma$, then $C(K)$ is a face of
$C(J)$;
\item For any $J\in \Sigma$ the map $C$ restricted on
$\{K\in \Sigma\mid K\preceq J\}$ is  an isomorphism of ordered sets
onto the set of faces of $C(J)$.
\end{enumerate}
For an integer $k$ with $0\leq k\leq n$ we set
\[ \sigmk=\{K\in \Sigma\mid \dim C(K)=k \}. \]
$w^\pm$ are maps $\sigmn\to \Z_{\geq 0}$.

A multi-fan is said complete, as in the case of simplicial multi-fans,
if it satisfies the condition stated in Definition in Section 2.
For $K\in\sigmk$ the projected multi-fan $\Delta_K$ and its degree are
also defined in a similar way as in the case of simplicial multi-fans.

By a triangulation of $\Delta$ we mean a simplicial multi-fan
$\delpr=(\sigmpr,\Cpr,w^{'\pm})$ in the same lattice $L$ related
to $\Delta$ in the following way:
\begin{enumerate}
\renewcommand{\labelenumi}{\alph{enumi})}
\item There is a bijection $\kappa:\sigmone\to\Sigma^{\prime(1)}$
satisfying
$C(\kappa (i))=C^\prime(\kappa(i))$ for each $\jpr \in \sigmpr$.
\item For each simplex $J^\prime \in \sigmpr$ there is an element
$J\in \Sigma$
such that $\Cpr(\Jpr)\subset C(J)$. Moreover, for each $J\in \Sigma$,
the collection $\{\Cpr(\Jpr)\mid \Jpr\in \sigmpr, \Cpr(\Jpr)\subset C(J)\}$
gives a subdivision of the cone $C(J)$.
We shall denote by $\rho(\Jpr)$ the minimal element $J\in \Sigma$ such that
$\Cpr(\Jpr)\subset C(J)$.
\item For $\Ipr\in \sigmprn$
\[ w^{\prime\pm}(\Ipr)=w^\pm(\rho(\Ipr)). \]
In particular $w^\prime(\Ipr)=w(\rho(\Ipr))$.
\end{enumerate}

Returning to general multi-fans we shall assume that every $J\in \Sigma$
is contained in some $I\in \sigmn$ hereafter.
Let $\V=\{v_i\}_{i\in \sigmone}$ be a set of non-zero vectors
$v_i\in L\cap C(i)$. A set of rational numbers
$\xi=\{d_i\}_{i\in \sigmone}$ is called \emph{$\Q$-Cartier} if there is an
element $u(I)\in L_\Q^*$ for each $I\in\sigmn$ such that
\begin{equation*}\label{eq:QCartier}
 \l u(I), v_i\r =d_i \quad \text{for $i\in I$}.
\end{equation*}
The pair $(\Delta,\V)$ is called \emph{$\Q$-Gorenstein} if there is an
element $u(I)\in L_\Q^*$ for each $I\in\sigmn$ such that
\begin{equation*}\label{eq:QGoren}
 \l u(I), v_i\r =1 \quad \text{for $i\in I$}.
\end{equation*}
When $\Delta$ is simplicial every $\xi$ is $\Q$-Cartier and
every pair $(\Delta,\V)$ is $\Q$-Gorenstein.

Take a triangulation $\delpr=(\sigmpr,\Cpr,w^\pm)$ of $\Delta$. Then
the collection $\xi$ determines a
$\Q$-divisor $\xipr=\sum_id_ix_i\in H_T^2(\delpr,\V)\otimes\Q$ on $\delpr$.

\begin{theo}\label{thm:QCartier}
Let $(\Delta,\V,\xi)$ be as above. If $\xi$ is $\Q$-Cartier, then the orbifold
elliptic genus $\hatvarst(\delpr,\V,\xi)$ does not depend on
$\sigmpr$. It depends only on $(\Delta,\V,\xi)$.
\end{theo}

\begin{coro}\label{coro:QGoren}
Let $(\Delta,\V)$ be a $\Q$-Gorenstein pair. Then the orbifold
elliptic genus $\hatvarst(\delpr,\V)$ does not depend on
$\sigmpr$ giving an invariant of $(\Delta,\V)$.
\end{coro}

\begin{proof}
Take triangulations $\delpr$ and $\delprpr$ of $\Delta$. Fix $I\in
\sigmn$ and put $\Sigma(I)=\{J\in \Sigma\mid J\subset I\}$. Then
$\Delta(I)=(\Sigma(I),C|\Sigma(I))$ determines a fan. Note that we
are neglecting the functions $w^\pm$ defined on $\Delta(I)$ for
the moment. We then define
\[ b_J(\delpr,\V,\xipr;\delprpr,\V,\xiprpr):=
 \sum_{\Jpr, \rho(\Jpr)=J}b_{\Jpr}(\delpr,\V,\xipr)
 -\sum_{\Jprpr, \rho(\Jprpr)=J}b_{\Jprpr}(\delprpr,\V,\xiprpr). \]
for $J\in \Sigma(I)$. We shall prove the following fact by
induction on $n$.
\vspace*{0.3cm}\\
$B_n^{flop}$ \hspace{1cm} $b_J(\delpr,\V,\xipr;\delprpr,\V,\xiprpr)=0$
for $\dim L=n$ and for all $J\in \Sigma(I)$.
\vspace*{0.3cm}\\
This will prove Theorem \ref{thm:QCartier} in view of \eqref{algn:BL} as in
the proof of Theorem \ref{thm:crepant1}.

The cases $n=1$ and $n=2$ is trivial since every cone is simplicial
in these cases.

Suppose $n\geq 3$. Let $\delpr(I)$ and $\delprpr(I)$ be the triangulations
of $\Delta(I)$ induced by $\delpr$ and $\delprpr$ respectively.
We construct fans
$\delprstar=(\sigmprstar,\Cprstar)$ and
$\delprprstar=(\sigmprprstar,\Cprprstar)$ from $\delpr(I)$ and $\delprpr(I)$
respectively in a similar way to the proof of Theorem \ref{thm:Bn}. Namely
\[ \sigmprstar=\sigmpr\cup\{\Jpr*(J_1\cdots J_k)\mid
\rho(\Jpr)\subset
J_1\subsetneqq \ldots\subsetneqq J_k \subsetneqq I\},\]
and
\[ \sigmprprstar=\sigmprpr\cup\{\Jprpr*(J_1\cdots J_k)\mid
\rho(\Jprpr)\subset
J_1\subsetneqq \ldots\subsetneqq J_k \subsetneqq I\}.\]
We set
\[ \partial\sigmprstar=\{(J_1\cdots J_k)\in \sigmprstar\},\quad
  \partial\sigmprprstar=\{(J_1\cdots J_k)\in \sigmprprstar\}. \]
They are isomorphic to the so-called order complex of the poset
$\Sigma(I)\setminus \Sigma(I)^n$ as simplicial complexes.
They will be identified with each other in the sequel. We have
\[ {\sigmprstar}^{(1)}={\sigmpr}^{(1)}\cup (\partial\sigmprstar)^{(1)} \
\text{where $(\partial\sigmprstar)^{(1)}=\{(J)\mid J\subsetneqq I\}$}, \]
and
\[ {\sigmprprstar}^{(1)}={\sigmprpr}^{(1)}\cup (\partial\sigmprprstar)^{(1)} \
\text{where $(\partial\sigmprprstar)^{(1)}=\{(J)\mid J\subsetneqq I\}$}. \]

Moreover $\Cprstar((J))=\Cprprstar((J))$ for
$(J)\in (\partial\sigmprstar)^{(1)}=(\partial\sigmprprstar)^{(1)}$.
Define a vector
$v_J\in \Cprstar((J))=\Cprprstar((J))$ as in the proof of Theorem \ref{thm:Bn}
and set
\[ \Vstar=\V\cup \{v_J\} \]
where  and $J$ ranges over $\sigmk,\ 0< k <n$.

The construction of $\delprstar$, $\delprprstar$ and $v_J$ depends on
an integer $m$. Similarly to the proof of
Theorem \ref{thm:Bn} it can be shown that these vectors $v_J$ satisfy
the property:
\begin{equation}\label{eq:vJinfty}
\text{The absolute value of $\l u,v_J\r$ tends to $\infty$ as
$m$ tends to $\infty$ for any $u\in L^*$}.
\end{equation}
Then the number $d_J$ is defined by
\begin{equation*}
 d_J=\l u^I(\xi), v_J\r,
\end{equation*}
and the $\Q$-divisor $\xiprstar$ by
\[ \xiprstar=\sum_{i'\in \sigmprone}d_{i'}x_{i'}+
 \sum_{(J)\in (\partial\sigmprstar)^{(1)}}d_Jx_J, \]
where $x_J$ is the basis element in $H_T^2(\delprstar,\Vprstar)$
corresponding to $(J)\in (\partial\sigmprstar)^{(1)}$.
The $\Q$-divisor $\xiprprstar$ is similarly defined.

We glue $\sigmprstar$ and $\sigmprprstar$ along the common boundary
$\partial\sigmprstar=\partial\sigmprprstar$ to obtain a multi-fan
$\delstar$. The functions $w^\pm$ are defined in such a way that
\begin{align*}
   w(J_*)=\begin{cases} 1, &  \text{for $J_*\in(\sigmprstar)^{(n)}$,} \\
           -1, &  \text{for $J_*\in\sigmstar^{(n)}$.} \end{cases}
\end{align*}
The multi-fan $\delstar$ is complete and the following equality holds.
\begin{equation}\label{eq:star2}
\deg((\delstar)_{J_*})=\begin{cases}
      +1, & J_*\in \sigmprstar\setminus \partial\sigmprstar, \\
      -1, & J_*\in \sigmprprstar\setminus \partial\sigmprprstar, \\
      \ \ 0, & J_*\in \partial\sigmprstar=\partial\sigmprprstar.
 \end{cases}
\end{equation}
The proof of these facts is similar to that of Theorem \ref{thm:Bn}.

The $\Q$-divisors $\xiprstar$ on $\sigmprstar$ and
$\xiprprstar$ on $\sigmprprstar$ define a $\Q$-divisor $\xistar$
on $\sigmstar$. Put $u=u^I(\xi)=\sum_{i\in I}d_i\u$ as before. Then
\[ \xistar=u \in H_T^2(\delstar,\Vstar)\otimes \Q. \]

We apply Theorem \ref{theo:varrigid} and get
\begin{equation}\label{eq:hatvarst9}
 \hatvarst(\delstar,\Vstar,\xistar)=0.
\end{equation}
On the other hand by \eqref{eq:star2} we have
\[\hatvarst(\delstar,\Vstar,\xistar)=
b_(\delpr,\V,\xipr;\delprpr,\V,\xiprpr)+b_2, \]
where
\[ b_2=\sum_{J\in \Sigma(I),J\not=I}\Biggl(
\sum_{\Jpr*K,\rho(\Jpr)=J}b_{\Jpr*K}(\delprstar,\Vstar,\xiprstar)
- \sum_{\Jprpr*K,\rho(\Jprpr)=J}
 b_{\Jprpr*K}(\delprprstar,\Vstar,\xiprprstar)\Biggr). \]
The term in the parenthesis for $K=\emptyset$ is equal to $0$ by
inductive assumption.
For $K\not=0$ the term tends to $0$ when $m$ tends to $\infty$
for any $u\in L^*$, as follows from a similar argument to the proof of
Theorem \ref{thm:Bn} using \eqref{eq:vJinfty}.
From this and \eqref{eq:hatvarst9} it follows that
$b_I(\delpr,\V,\xipr;\delprpr,\V,\xiprpr)$ must be equal to zero.
Together with inductive assumption this proves that
\[ b_J(\delpr,\V,\xipr;\delprpr,\V,\xiprpr)=0 \]
for all $J\in\Sigma(I)$.
Thus $B_n^{flop}$ holds and Theorem \ref{thm:QCartier} is proved.
\end{proof}


\providecommand{\bysame}{\leavevmode\hbox to3em{\hrulefill}\thinspace}

\end{document}